\documentclass[11pt]{amsart}
\usepackage{amsfonts,amssymb,amsmath,amsthm,bbm}

\usepackage[english]{babel}

\def\a{\mathfrak{a}}

\newcommand{\RR}{{\bf R}}
\newcommand{\R}{{\bf R}}

\newcommand{\ZZ}{\mathbb{Z}}

\newtheorem{theorem}{Theorem}[section]
\numberwithin{equation}{section}
\newtheorem{definition}[theorem]{Definition}

\newtheorem{lemma}[theorem]{Lemma} 
\newtheorem{proposition}[theorem]{Proposition}
\newtheorem{corollary}[theorem]{Corollary}
\newtheorem{remark}[theorem]{Remark}
\newtheorem{conjecture}{Conjecture}

\newtheorem{example}[theorem]{Example}
\newtheorem{acknowledgment}{Acknowledgment}

\begin{document}

\title{Potential kernels for radial Dunkl Laplacians}


\author{Piotr Graczyk}
\address{P.{} Graczyk, LAREMA, UFR Sciences, Universit\'e d'Angers, 2 bd Lavoisier, 49045 Angers cedex 01, France}
\email{graczyk@univ-angers.fr}

\author{Tomasz Luks}
\address{T.{} Luks, Institut f\"ur Mathematik, Universit\"at Paderborn, Warburger Strasse 100, D-33098 Paderborn, Germany}
\email{tluks@math.uni-paderborn.de}

\author{Patrice Sawyer}
\address{P.{} Sawyer, Department of Mathematics and Computer Science, Laurentian University, Sudbury, Canada}
\email{psawyer@laurentian.ca}



\begin{abstract}
We derive two-sided bounds for the Newton and Poisson kernels
of the $W$-invariant Dunkl {Laplacian} {in geometric complex case when} the multiplicity $k(\alpha)=1$ {i.e. for flat complex symmetric spaces}. For the invariant Dunkl-Poisson kernel $P^W(x,y)$, the estimates are
\begin{align*}
P^W(x,y)\asymp \frac{P^{\R^d}(x,y)}{\prod_{\alpha > 0 \ }|x-\sigma_\alpha y|^{2k(\alpha)}},
\end{align*}
where the $\alpha$'s are the positive roots of a root system acting in $\R^d$, the $\sigma_\alpha$'s are the corresponding symmetries and $P^{\R^d}$ is the classical Poisson kernel in ${\R^d}$.  Analogous bounds are proven for the Newton kernel when $d\ge 3$.

The same estimates are derived in the rank one direct product case $\ZZ_2^N$ and conjectured for general $W$-invariant Dunkl processes.

As an application, we get a two-sided bound for the Poisson and Newton kernels of the classical Dyson Brownian motion and of the Brownian motions in any Weyl chamber.
\end{abstract}

\maketitle
\selectlanguage{english}

\section{Introduction}

Dunkl differential-difference operators were discovered by Dunkl \cite{D} in the late 1980's as a crucial tool to study Calogero-Moser-Sutherland mechanical particle systems. Since then,
Dunkl analysis was intensely developed and its main achievements are gathered in \cite{DJ, DuX, Ro, RV}. Together with their trigonometric counterparts, which were introduced by I. Cherednik in 1991 \cite{C}, Dunkl operators provide an extension of the commuting families of differential operators associated to Riemannian symmetric spaces. They have become an essential tool in the modern analysis associated with root systems.

The main results of the paper are two-sided bounds for the Poisson and Newton  kernels (when $d\ge 3$) of 
$W$-invariant Dunkl Laplacians (also called radial or symmetrized Dunkl Laplacians),  in the geometric complex case when the multiplicity $k(\alpha)=1$ (i.e. for flat complex symmetric spaces) and in the rank one direct product case $\ZZ_2^N$ with any multiplicities. These estimates have an elegant form
\begin{align*}
 {\mathcal K}^W(x,y)\asymp \frac{{\mathcal K}^{\R^d}(x,y)}{\prod_{\alpha > 0 \ }|x-\sigma_\alpha y|^{2k(\alpha)}},
\end{align*}
 where ${\mathcal K}^{\R^d}$ is a classical kernel and ${\mathcal K}^W(x,y)$ its radial Dunkl counterpart. Analogous estimates are  conjectured for the general $W$-invariant Dunkl Laplacians.
 
 Precise estimates of Poisson and Newton kernels (as well as of the Green function) of local and non local operators have been of interest for many years and constitute an important part of the modern potential theory, see e.g. \cite{CheS, CrZ,  Ku, Wid, Z1, Z2} and references therein. In this context, the only case treated so far in the Dunkl setting is the root system of rank one, see \cite{GLR}. The main difficulty for general root systems is the lack of explicit formulas (or suitable estimates) for the intertwining operator (see Section \ref{sec:prelim} for its connection with potential kernels). Our methods for the geometric complex case rely on the so-called alternating sums (see Proposition \ref{altern}) and allow us to handle all root systems. In the rank one direct product case, we make use of available formulas for the intertwining operator, see Section \ref{sec:product}.
 
Our results contribute significantly to the further development of the potential theory of Dunkl operators. In particular, they will be essential to study the estimates of the Green function, not treated in this paper. The study of two-sided bounds of potential kernels has motivations and immediate applications in partial differential equations and implies estimates of solutions of important PDEs, in particular of Dunkl-harmonic functions (see e.g. \cite{Chr, Gallardo1, Gallardo2, GLR} for the latter).

Our paper is organized as follows. In Section \ref{sec:prelim}, we recall basic facts from Dunkl analysis and we formulate our main results. In Section \ref{sec:roots}, we prove some useful formulas involving roots. In Section \ref{Dunkl setting}, we present results on the Poisson kernel which hold without any restriction on the multiplicity function $k$. In Section \ref{sec:complex} and Section \ref{sec:Newton}, we prove our main results for the Poisson and Newton kernel, respectively, in the geometric complex case. In Section \ref{sec:product}, we discuss the rank one direct product case. The applications  to the important stochastic processes Dyson Brownian motions and stochastic particle systems are discussed in Section \ref{stochastic}.

\section{Preliminaries}\label{sec:prelim}
\subsection{Basic facts on root systems and Dunkl operators}

Let $\Phi$ be a root system in $\R^d$ (equipped with the usual scalar product and Euclidean norm $|\cdot|$), and let $W$ be the associated finite reflection
group, called Weyl group. We fix a non negative multiplicity function $k$ on $\Phi$, i.e. $k: \Phi\to [0, \infty)$ is $W$-invariant, and let $\Phi_+$ denote an (arbitrary) positive subsystem of $\Phi$.

The root system $\Phi$ is assumed to be crystallographic {(needed in Lemma \ref{XWY})} but $W$ is not required to be effective, i.e. $\text{span}_\R\, \Phi$ may be a proper subspace of $\R^d$. The dimension of $\text{span}_\R\,\Phi$ is called the rank of $\Phi$. 

An important example is 
\begin{align*}
\Phi=A_{r-1} = \{ \pm(e_i-e_j): 1\leq i < j \leq r\}\subset \R^d, \qquad d\ge r
\end{align*}
with $W=S_r$, the symmetric group in $r$ elements. 
Note that $A_{r-1}$ is also considered on $\R^{r-1}=\R^r\cap \{ x|\ \sum_{i=1}^r x_i=0 \}$.

We will assume that $|\alpha|\geq 1$ for all roots $\alpha$. We denote by 
$A_\alpha$ the root vectors, i.e. $\langle x,A_\alpha\rangle=\alpha(x)$ for all $x\in\a$. We denote $(\alpha_i)_{i=1,\ldots, r} $ a system of simple positive roots.

The (rational) Dunkl operators associated with $\Phi$ and $k$ are given by
\begin{align*}
T_\xi f(x) = \partial_\xi f(x) + \sum_{\alpha\in \Phi_+} 
 k(\alpha)\,\langle\alpha, \xi\rangle\, 
 \frac{f(x) - f(\sigma_\alpha x)}{\langle\alpha, x\rangle},
 \quad\xi\in \R^d
\end{align*}
where $\partial_\xi$ is the derivative in the direction of $\xi$.

For fixed $\Phi$ and $k$, these operators commute. Moreover, there is a unique linear isomorphism $V_k$ on the space of polynomial functions in $d$ variables, called the intertwining operator, which preserves the degree of homogeneity, is normalized by $V_k(1)=1$ and intertwines the Dunkl operators with the usual partial derivatives:
\begin{align*}
T_\xi V_k = V_k \partial_\xi \quad \text{ for all } \xi \in \R^d.
\end{align*}
The Dunkl Laplacian is defined by 
\begin{align*}
\Delta_k := \sum_{i=1}^ d T_{\xi_i}^2 
\end{align*}
with an arbitrary orthonormal basis $(\xi_i)_{ 1\leq i \leq d}$ of $\R^d$. 
Then, for $f\in {\mathcal C}^2( \R^d)$,
\begin{align*}
\Delta_k f(x)=\Delta f(x)+\sum_{\alpha\in \Phi_+}k(\alpha)\left(\frac{\langle\nabla f(x),\alpha\rangle}{\langle\alpha,x\rangle}-
\frac{|\alpha|^2}{2}\frac{f(x)-f(\sigma_\alpha x)}{\langle\alpha,x\rangle^2}\right),
\end{align*}
where $\Delta$ is the usual Laplacian on $\R^d$. 
The restriction of $\Delta_k$ to $W$-invariant functions $f\in {\mathcal C}^2_W( \R^d)$ is called the $W$-invariant (or radial)
Dunkl Laplacian and is denoted by $\Delta^W$. We have
\begin{align}\label{DeltaW}
\Delta^W f(x)=\Delta f(x)+\sum_{\alpha\in \Phi_+}k(\alpha)\frac{\langle\nabla f(x),\alpha\rangle}{\langle\alpha,x\rangle},\qquad f\in {\mathcal C}^2_W( \R^d).
\end{align}

For $x\in \R^d$, denote by $C(x)$ the convex hull of the Weyl group
orbit $W\cdot x$ of $x$ in $\R^d.$
The intertwining operator $V_k$ has the integral representation \cite{Ro}
\begin{align}\label{intertwiner}
V_kf(x)=\int_{C(x)}\, f(z)\, d\mu_x^k(z),
\end{align}
where $\mu_x^k$ is a probability measure on $C(x)$, called R\"osler measure, and $f$ is a ${\mathcal C}^1$ function. 
We put 
\begin{align*}
\kappa := \sum_{\alpha\in \Phi_+}\,k(\alpha)\,
\end{align*}
and define the weight function $\omega_k$ on $\R^d$ by
\begin{align*}
\omega_k(x):=\prod_{\alpha\in \Phi_{+}}\left|\langle\alpha,x\rangle\right|^{2k(\alpha)}.
\end{align*}

Let $\a^+$ be the positive Weyl chamber associated with $\Phi_+$. 
We denote $\pi(x)=\prod_{\alpha\in \Phi_{+}}\langle\alpha,x\rangle$.

\subsection{Invariant Poisson and Newton kernels}

\subsubsection{$W$-invariant Poisson kernel}

For the convenience of the reader, 
we recall here main facts about the $W$-invariant Poisson kernel $P^W(x,y)$ of the unit ball,
presented in \cite{PGPS}.

We denote by $B:=B(0,1)$ the unit ball in $\R^d$ and by $S:=\partial B$ the unit sphere. Denote $B^+:=B\cap\overline{\a^+}$ and $S^+:=S\cap\overline{\a^+}$.
First recall the definition of the Dunkl Poisson kernel.

The Dunkl Poisson kernel $P_k(x,y)$ of $B$ is the continuous function on $B\times S$ which solves the Dirichlet problem for the 
Dunkl Laplacian operator $\Delta_k$ {on $\overline B$}
 namely, for any $f\in {\mathcal C}(S)$,
 \begin{align*}
 u(x)=\int_S\,P_k(x,y)\,f(y)\,{\pi^2}(y)\,dy
 \end{align*}
 is a solution of the equation $\Delta_k\,u(x)=0$
{on $B$, with $u\in {\mathcal C}^2(B)\cap {\mathcal C}(\overline{B})$}
 and $u(y)=f(y)$ for $y\in S$.
The existence and uniqueness of the Dunkl Poisson kernel $P_k(x,y)$
was shown by Dunkl \cite{Du1} for spherical polynomials
$f$ and extended in \cite{MYoussfi} to any continuous $f$ on $S$.


The Dunkl Poisson kernel is given by Dunkl formula (refer to \cite{DuX})
\begin{align}\label{PoissonDunkl}
P_k(x,y)=\frac{2^{2\,\kappa}\,(d/2)_\kappa}{\pi(\rho)\,|W|\,w_d}\,V_k\left[\frac{1-|x|^2}{(1-2\langle x,\cdot\rangle +|x|^2)^{\kappa+
d/2}}\right](y),\ \ x\in B,~ y\in S.
\end{align}
Note that we are using a different normalization from \cite{DuX} for consistency between  \eqref{PoissonDunkl}, \eqref{PoissonComplexX} and Remark \ref{average}.

\begin{remark}\label{scalprod}
We will use repeatedly the fact that $\langle x,y\rangle\geq \langle x,w\,y\rangle$ whenever $x$, $y\in\overline{\a^+}$ and $w\in W$ (\cite[Theorem 2.12, Ch.{} VII]{Helgason1}).
\end{remark}

The following result is well known. We include the proof for convenience.
\begin{lemma}\label{projection}
The map $\pi_{\a^+}$ which sends an element $x\in\a$ to the unique element $x^+=w\,x\in\overline{\a^+}$ with $w\in W$ is Lipschitz.
\end{lemma}
\begin{proof}
If $x$, $y\in \overline{\a^+}$ and $w\in W$ then $|x-w\,y|^2=|x-y|^2+2\,\langle x,y-w\,y\rangle\geq |x-y|^2$ since $\langle x,y-w\,y\rangle\geq0$. Hence, if $x^+=w_1\,x$ and $y^+=w_2\,y$ then 
\begin{align*}
|x-y|^2=|(w_1\,x)-w_1\,w_2^{-1}\,(w_2\,y)|^2=|x^+-w_1\,w_2^{-1}\,y^+|^2\geq |x^+-y^+|^2.
\end{align*}
\end{proof}

 \begin{definition}\label{def:PW}
 The $W$-invariant Poisson kernel $P^W(x,y)$ of $B$ is the $W\times W$ invariant continuous function on $B\times S$, which solves the Dirichlet problem for the operator $\Delta^W$ on $B$, where $\Delta^W$ is the radial Dunkl Laplacian \eqref{DeltaW}.
 Namely, for any $f\in {\mathcal C}_W(S)$,
 \begin{align*}
 u(x)=\int_S\,P^W(x,y)\,f(y)\,\pi^2(y)\,dy
 \end{align*}
is a solution of the equation $\Delta^W\,u(x)=0$ on $B$, with $u\in {\mathcal C}_W^2(B)\cap {\mathcal C}_W(\overline{B})$ and $u(y)=f(y)$ for $y\in S$.
 \end{definition} 
The next proposition will show that the Weyl-invariant Poisson kernel is uniquely determined by the above definition. 

\begin{remark}
The proof of the uniqueness of the symmetrized Poisson kernel carries through in the same manner in general.
\end{remark}

 \begin{proposition}\label{altern}
In the complex case, the Poisson kernel of the open unit ball $B$ is given for $x\in B$ and $y\in\partial B$ by
\begin{align}\label{PoissonComplexX}
P^W(x,y) =\frac{1-|x|^2}{|W|\,w_d\,\pi(x)\,\pi(y)} \,\sum_{w\in W}\,\frac{\epsilon(w)}{|x-w\cdot y|^{d}}.
\end{align}
\end{proposition}
\begin{proof}
The derivation of \eqref{PoissonComplexX} is based on the properties of the regular Euclidean Poisson kernel $\frac{1-|x|^2}{w_d}\frac{1}{|x-y|^{d}}$
and on the formula
\begin{equation}\label{operator}
\Delta^W f= \pi^{-1}\, \Delta^{\R^d} (\pi\, f),
\end{equation}
(see \cite[Chap.{} II, Theorem 5.37]{Helgason2}).

It is straightforward that this Weyl-invariant kernel solves the Dirichlet problem on the ball $B$.

Suppose that there is another Weyl-invariant kernel $\tilde{P}$ which also solves the Dirichlet problem on the ball $B$. We have 
\begin{align*}
\int_S\,(P^W(x,y)-\tilde{P}(x,y))\,g(y)\,\pi^2(y){dy}=0
\end{align*}
for every continuous Weyl-invariant function $g$ on $S$. There is a one to one correspondence between the continuous Weyl-invariant functions on $S$ and those on $S^+$. 
 It suffices to note that the map $x\mapsto \pi_{\a^+}(x)$ which projects $x$ on the unique $y=w\,x\in\overline{\a^+}$
is continuous: this is a direct consequence of Lemma \ref{projection}. Hence, for every continuous function $g$ on $S^+$ and for every $x\in B$, we have
\begin{align*}
\int_{S\cap\overline{\a^+}}\,(P^W(x,y)-\tilde{P}(x,y))\,g(y)\,\pi^2(y)dy=0.
\end{align*}
By the properties of the Lebesgue integral, we can therefore conclude that $P^W(x,y)-\tilde{P}(x,y)=0$ for every $y\in S^+$ and, hence, for every $y\in S$.
\end{proof}

\begin{remark}\label{average}
Because of the uniqueness of $P^W$, we can conclude that
\begin{align}
P^W(x,y)=\frac{1}{|W|}\sum_{w\in W}\,P_k(x,w\,y)\label{summation}
\end{align}
where $P_k$ is the Dunkl Poisson kernel and $k=1$ since it satisfies the Definition \ref{def:PW}.

Observe that the right hand term in \ref{summation} is not only Weyl-invariant in $y$ but also in $x$.  This follows from the fact that in proving the uniqueness in Proposition \ref{altern}, we only used Weyl-invariance in $y$.
\end{remark}
 

\subsubsection{$W$-invariant Newton kernel}
Recall that the heat kernel of $\Delta_k$ is given by (see \cite{Ro2})
\begin{align*}
p^k_t (x,y) =
\frac{M_k}{t^{\gamma+ d/2}}\,e^{-(|x|^2+|y|^2)/4t} 
E_k\bigl(\frac{x}{\sqrt{2t}}, \frac{y}{\sqrt{2t}}\bigr),
\end{align*}
where 
\begin{align*}
M_k = 2^{-\gamma-d/2}\bigl(\int_{\RR^d} e^{-|x|^2/2}\omega_k(x)dx\bigr)^{-1}
\end{align*}
and $E_k(x,y)=V_k(e^{\langle\cdot,y\rangle})(x)$ denotes the Dunkl kernel. The Newton kernel of $\Delta_k$ is defined by
\begin{align*}
N_k(x,y)=\int_0^\infty\,p^k_t (x,y)\,dt.
\end{align*}
A general formula for $N_k(x,y)$ involving the intertwining operator, analogous to \eqref{PoissonDunkl}, was proven in \cite{Gallardo2}. It is given by
\begin{align}
N_k(x,y)
&=\frac{2^{2\,\gamma}\,((d-2)/2)_\gamma}{|W|\,(d-2)\,w_d\,\pi(\rho)}\,V_k\left({(|y|^2-2\,\langle x,\cdot\rangle+| x|^2)^{(2-d-2\,\gamma)/2}} \right)(y)\label{NewtonDunkl}
\end{align}
(as in the case of the Poisson kernel, we are using a slightly different normalization).
 
The $W$-invariant Newton kernel serves as the inverse of the operator $\Delta^W$. It solves the problem $\Delta^W\,u=f$ where $f$ is given and $|u(x)|\to0$ as $x\to\infty$. It is defined by 
\begin{align*}
N^W(x,y)=\int_0^\infty\,p^W_t(x,y)\,dt,
\end{align*}
where
\begin{align*}
p^W_t(x,y)=\frac{1}{|W|}\,\sum_{w\in W}\,p^k_t (x,w\,y)
\end{align*}
is the heat kernel of $\Delta^W$. We then have 
\begin{align*}
N^W(x,y)=\frac{1}{|W|}\,\sum_{w\in W}\,N_k(x,w\,y).
\end{align*}
Using similar arguments as for the Poisson kernel, we also have the alternating formulas
\begin{align*}
N^W(x,y)&=\frac{1}{4\,\pi\,\pi(x)\,\pi(y)}\,\sum_{w\in W}\,\epsilon(w)\,\ln|x-w\cdot y|~~ \hbox{when $d=2$},\\
N^W(x,y)&=\frac{1}{|W|\,(2-d)\,w_d\,\pi(x)\,\pi(y)}
\,\sum_{w\in W}\,\frac{\epsilon(w)}{|x-w\, y|^{d-2}}~~ \hbox{when $d\geq3$},\nonumber
\end{align*}
in the geometric complex case. In particular,
\begin{align}
N^W(x,0)&=\int_0^\infty\,p^W_t(x,0)\,dt
=\frac{1}{|W|\,2^d \,\pi^{d/2}\,\pi(\rho)} \,\int_0^\infty\,t^{-\frac{d}2-\gamma} \,e^{-|x|^2/(4\,t)}\,dt\nonumber\\
&=\frac{\int_0^\infty\,u^{-\frac{d}{2}-\gamma} \,e^{-1/(4\,u)}\,du}{|W|\,2^d \,\pi^{d/2}\,\pi(\rho)} 
\,\frac{1}{|x|^{d-2+2\,\gamma}}=\frac{2^{2\,\gamma-2}\,\Gamma(d/2+\gamma-1)}{|W|\,\pi^{d/2}\,\pi(\rho)} 
\,\frac{1}{|x|^{d-2+2\,\gamma}}.\label{interesting}
\end{align}

\subsection{Main results of the paper}
{We write $f\asymp g$ when there are constants $0<C_1\leq C_2$ depending only on the dimension, on the choice of the root system
and of the multiplicity function $k$,
such that, on the common domain of $f$ and $g$, we have $C_1\,f\leq g\leq C_2\,f$.}

\begin{conjecture}\label{TL} 
For $x\in B^+$ and $y\in S^+$ we have
\begin{align*}
P^W(x,y)\asymp \Omega(x,y):=\frac{1-|x|^2}{|x-y|^d\prod_{\alpha\in \Phi_+} |x-\sigma_\alpha y|^{2k(\alpha) }}.
\end{align*}
\end{conjecture}
\noindent In the case of the Newton kernel, the conjecture is the following.
 
\begin{conjecture}\label{TLN} 
For $x, y\in \overline{\a^+}$ and $d\geq 3$, we have
\begin{align*}
N^W(x,y)\asymp \frac{1}{|x-y|^{d-2}\,\prod_{\alpha\in \Phi_+}|x-\sigma_\alpha y|^{2\,k(\alpha)}}.
\end{align*}
\end{conjecture}

\noindent We now formulate our main results.

\begin{theorem}\label{main k1}
Conjectures \ref{TL} and \ref{TLN} hold for all complex root systems. 
\end{theorem}

\begin{theorem}\label{main pr}
Conjectures \ref{TL} and \ref{TLN} hold for the direct product rank 1 case. 
\end{theorem}

\noindent We also give some results supporting the conjectures in the general Dunkl setting in Section \ref{Dunkl setting}.  Note also that \eqref{interesting} is consistent with Conjecture \ref{TLN}.

\section{Analysis with roots}\label{sec:roots}

\subsection{Some simple formulas}
In this section we give some simple and useful formulas involving roots.

For $\alpha\in \Phi$ define 
\begin{align*}
\phi_\alpha(x,y):=\max\{|x-y|, d(y,H_\alpha)\},
\end{align*} 
where $H_\alpha$ denotes the hyperplane perpendicular to $\alpha$ and $d$ is the usual Euclidean distance.
The following simple properties will be often used.

\begin{lemma}\label{basic simple}~
\begin{itemize}
\item[(i)] Let $y\in \overline{\a^+}$ and $\alpha$ be a positive root. Then 
\begin{align*}
d(y,H_\alpha)=\alpha(y)/|\alpha|.
\end{align*}
\item[(ii)] For any $x,y\in{\a}$ we have 
\begin{align*}
|x-\sigma_\alpha y|^2=|x-y|^2+C\,\alpha(x)\alpha(y)
\end{align*}
where $C=4/|\alpha|^2$.
\end{itemize}
\end{lemma}

\begin{proof}~
\begin{itemize}
\item[(i)] Write $y=y_0+b\,A_\alpha$ where $y_0\in H_\alpha$ and $\langle x,A_\alpha\rangle=\alpha(x)$ for all $x\in\a$. Observe that $b=\alpha(y)/|\alpha|^2$. Then, for $z\in H_\alpha$ we have 
\begin{align*}
|y-z|^2=|y_0-z|^2+b^2|A_\alpha|^2
\end{align*}
which attends its minimum when $z=y_0$. Hence, 
$\displaystyle d(y,H_\alpha)=b\,|\alpha|=\frac{\alpha(y)}{|\alpha|}$.
\item[(ii)] Follows by direct computation.
\end{itemize}
\end{proof}
There is an equivalent formulation for $\phi_\alpha(x,y)$.

\begin{lemma}\label{formul}
Let $x$, $y\in \overline{\a^+}$. We have
\begin{align*}
\phi_\alpha(x,y)\asymp |x-\sigma_\alpha(y)|.
\end{align*}
\end{lemma} 

\begin{proof}
 Indeed, 
\begin{align*}
|x-\sigma_\alpha y|^2\leq |x-y|^2+|y-\sigma_\alpha y|^2\asymp\phi_\alpha(x,y)^2
\end{align*}
while $|x-\sigma_\alpha y|^2\geq |x-y|^2$ and 
\begin{align*}
|x-\sigma_\alpha y|^2=|\sigma_\alpha x-y|^2\geq d(y,H_\alpha)^2
\end{align*}
(geometric argument: the straight line joining $y$ with $\sigma_\alpha x$ crosses $H_\alpha$; if $\ell(t)=ty+(1-t)\,\sigma_\alpha x$ then $\alpha(\ell(0))=\alpha(y)\geq0$ while $\alpha(\ell(1))=-\alpha(x)\leq 0$).
\end{proof}

\begin{remark}\label{either}
Lemma \ref{formul} shows that in proving Conjectures \ref{TL} and \ref{TLN}, one can replace $|x-\sigma_\alpha\,y|$ by $\phi_\alpha(x,y)$.  We will do this without further mention.
\end{remark}

\begin{remark}\label{XandY}
The same argument shows that 
\begin{align*}
|x-\sigma_\alpha y|^2=|y-\sigma_\alpha x|^2\asymp\max\{|x-y|^2,d(x,H_\alpha)^2\},
\end{align*}
so 
\begin{align*}
{\max\{|x-y|,d(x,H_\alpha)\}\asymp\phi_\alpha(x,y)= \max\{|x-y|,d(y,H_\alpha)\}}.
\end{align*}
\end{remark}

\begin{proposition}\label{XWY}
Let $\alpha_i$ be the simple roots and let $A_{\alpha_i}$ be such that $\langle x,A_{\alpha_i}\rangle=\alpha_i(x)$ for $x\in\a$.
Suppose $x,y\in\a^+$ and $w\in W\setminus\{id\}$. Then we have
\begin{align}\label{CL}
y-w\,y=\sum_{i=1}^r \,2\,\frac{a_i^w(y)}{|\alpha_i|^2}\,A_{\alpha_i}
\end{align}
where $a_i^w$ is a linear combination of positive simple roots with non-negative integer coefficients for each $i$.
\end{proposition}

\begin{example}
Let us first illustrate the proposition on an example. Consider $w=\sigma_{\alpha_2}\,\sigma_{\alpha_1}$.
We have
\begin{align*}
y-w\,y= \frac{2}{|\alpha_1|^2}\alpha_1(y)A_{\alpha_1} + \frac{2}{|\alpha_2|^2}\,\left(
\frac{-2\langle \alpha_1,\alpha_2\rangle }{|\alpha_1|^2}\alpha_1(y)+\alpha_2(y)\right)\,A_{\alpha_2}
\end{align*}
In particular, for the $A_n$ root system in $\R^d,\ d\ge n$ and the permutation $w=(123)\in W$, we have
$w=\sigma_{\alpha_2}\sigma_{\alpha_1}$and, with the normalization
$|\alpha_i|^2=2$, we get
\begin{align*}
y-wy= \alpha_1(y)A_{\alpha_1} + (
\alpha_1(y)+\alpha_2(y))A_{\alpha_2}.
\end{align*}
\end{example}

\begin{proof}
We first prove by induction on the length of the minimal representation of $w$ as a product of reflections by simple roots that \eqref{CL} holds where $a_i^w(y)$ is a linear combination of simple roots with integer coefficients for each $i$. 

If the length of $w$ is 1 then {$w=\sigma_{\alpha_i}$ for a simple root $\alpha_i$ and}
\begin{align*}
y-\sigma_{\alpha_i}\,y=2\,\frac{\alpha_i(y)}{|\alpha_i|^2}\,A_{\alpha_i}
\end{align*}
and the result holds in that case {with $a_i^w(y)=\alpha_i(y)$}.

Let us assume that the result holds for $k\geq 1$ and let $w=\sigma_{\alpha_j}\,w_0$ where $|w_0|=k$.
\begin{align*}
y-\sigma_{\alpha_j}\,w_0y
&=\sigma_{\alpha_j}\,(y-w_0y)+y-\sigma_{\alpha_j}y
=\sigma_{\alpha_j}\,\Bigl(\sum_{i=1}^r\,a_i^{w_0}(y)\,\frac{2A_{\alpha_i}}{|\alpha_i|^2}\Bigl)+2\,\frac{\alpha_j(y)}{|\alpha_j|^2}\,A_{\alpha_j}\\
&=\sum_{i=1}^r\,a_i^{w_0}(y)\,\sigma_{\alpha_j}\Bigl(\frac{2A_{\alpha_i}}{|\alpha_i|^2}\Bigl)+2\,\frac{\alpha_j(y)}{|\alpha_j|^2}\,A_{\alpha_j}\\
&=\sum_{i=1}^r\,a_i^{w_0}(y)\,\frac{2A_{\alpha_i}}{|\alpha_i|^2}
{-\sum_{i=1}^r\,a_i^{w_0}(y)\,2\,\frac{\langle\alpha_j,\alpha_i\rangle}{|\alpha_i|^2}\,2\,\frac{A_{\alpha_j}}{|\alpha_j|^2} }+2\,\frac{\alpha_j(y)}{|\alpha_j|^2}\,A_{\alpha_j}.
\end{align*}
Note that $2\,\frac{\langle\alpha_j,\alpha_i\rangle}{|\alpha_i|^2}$ is an integer since $\Phi$ is crystallographic.

We now need to prove that $a_i^w(y)$ is a linear combination of simple roots with non-negative integer coefficients for each $i$.

Suppose that this is not the case. To fix matters, we can assume that $a_1$ contains a term, say $b_2\,\alpha_2$, with $b_2<0$. Choose $x$, $y\in \overline{\a^+}$ such that $\alpha_2(y)=t>0$, $\alpha_k(y)=0$ for $k\not=2$ and $\alpha_1(x)=t>0$, $\alpha_k(x)=0$ for $k\not=1$. Then using \eqref{CL},
\begin{align*}
\langle x,y-wy\rangle =b_2\alpha_2(y)\alpha_1(x)=b_2t^2<0
\end{align*}
which cannot be true (refer to Remark \ref{scalprod}). The result follows.
\end{proof}

\begin{corollary}\label{cor:refdist}
There exists $C>0$ such that for every $x$, $y\in \a^+$ and $w\in W\setminus\{id\}$
\begin{itemize}
 \item[(i)] $|x-w\,y|^2\geq |x-y|^2+C\,\min_i\,\alpha_i(x)\,\min_j\alpha_j(y)$,
 \item[(ii)] $|x-w\,y|^2\geq\max\left\{\min\{d(x,H_\alpha)^2\colon\alpha>0\},\min\{d(y,H_\alpha)^2\colon\alpha>0\}\right\}$.
\end{itemize}
\end{corollary}

\begin{proof}
Let the notation be as in the Proposition \ref{XWY}.
By \eqref{CL} and \eqref{useful} we have
\begin{align*}
|x-wy|^2=|x-y|^2+2\,\langle x,y-wy\rangle
=|x-y|^2+2\sum_{i=1}^r\frac{a_i^w(y)\alpha_i(x)}{|\alpha_i|^2},
\end{align*}
and the first inequality follows since not all $a_i^w(y)$ are nul.

The second inequality follows from a simple geometric argument since $x$ and $wy$ are not in the same Weyl chamber, the straight line between $x$ and $wy$ must cross a wall of the Weyl chamber (similarly for $y$ and $w^{-1}\,{x}$).
\end{proof}

\begin{remark}\label{Up}
The linear combination $a_i^w$ only depends on $w\in W$ and therefore the largest coefficient $M$ of a root appearing in any $a_j^w$, for any $j$ or $w$, is finite. 
\end{remark}

\begin{corollary}\label{KillMaxSameWall}
For $z\in\a$, let $\Phi_z=\{\alpha\in\Phi\colon \alpha(z)=0\}$.
Let $x$, $y\in\overline{\a^+}$ such that $\Phi_x\subseteq \Phi_y$.
Suppose that $w\not\in W_y=\{w\in W\colon w\,y=y\}$. Then
\begin{align*}
\langle x,y \rangle > \langle x,wy \rangle.
\end{align*}
\end{corollary}

\begin{proof}
Recall that by Proposition \ref{XWY}
\begin{align*}
y-w\,y=\sum_{i=1}^r \,2\,\frac{a_i^w(y)}{|\alpha_i|^2}\,A_{\alpha_i}
\end{align*}
We first show that there exists $i_0$ such that $\alpha_{i_0}(y)>0$ and $a_{i_0}^w(y) >0$. Suppose by contradiction that this is not true, so
\begin{align*}
y-w\,y=\sum_{\alpha_i(y)=0} \,2\,\frac{a_i^w(y)}{|\alpha_i|^2}\,A_{\alpha_i}
\end{align*}
Then 
\begin{align*}
 \langle y, y-w\,y \rangle=\sum_{\alpha_i(y)=0} \,2\,\frac{a_i^w(y)}{|\alpha_i|^2}\,\alpha_i(y)=0,
\end{align*}
so $\langle y, y \rangle=\langle y, w\,y \rangle$. As $\|wy\|=\|y\|$, this is only possible if $y=wy$.

Since $\alpha_{i_0}\not\in \Phi_y\supseteq\Phi_x$, it follows that 
\begin{align*}
\langle x, y-w\,y \rangle=\sum_{i=1}^r \,2\,\frac{a_i^w(y)}{|\alpha_i|^2}\,\alpha_i(x)\ge 2\,\frac{a_{i_0}^w(y)}{|\alpha_{i_0}|^2}\,\alpha_{i_0}(x) >0.
\end{align*}
\end{proof}

\begin{corollary}\label{alpha k in ak}
Let $y\in \overline{\a^+}$ and $w\in W$. Consider the decomposition \eqref{CL} of $y-wy$. If $a_k^w(y)\not= 0$
then $\alpha_k$ appears in $a_k^w$, i.e. 
$a_k^w=\sum_{i=1}^r n_i \alpha_i$ with $n_k>0$.
\end{corollary}

\begin{proof}
Suppose that 
$a_k^w=\sum_{i=1}^r n_i \alpha_i$ with $n_k=0$. Consider $x_0$, $y_0$ with $\Phi_{x_0}=\Phi_{y_0}$ generated by $\{\alpha_i\colon i\not= k\}$. 
Then, for we have
\begin{align*}
\langle x_0,y_0-w\,y_0\rangle=\sum_{i=1}^r \,2\,\frac{a_i^w(y_0)}{|\alpha_i|^2}\,\alpha_i(x_0)= 2\,\frac{a_k^w(y_0)}{|\alpha_k|^2}\,\alpha_k(x_0)=0
\end{align*}
since $a_k^w(y_0)=\sum_{i\not=k} n_i \alpha_i(y_0)=0$.
By Corollary \ref{KillMaxSameWall}, we see that $w\in W_{x_0}$ with $W_{x_0}$ 
generated by $\sigma_{\alpha_i}$, $i\not =k$. This implies that
$A_{\alpha_k}$ does not appear in the decomposition 
 \eqref{CL} of $y-wy$, thus, equivalently that $a_k^w=0$.
\end{proof}

\begin{remark}\label{C}
There exist constants $C_1$ and $C_2$ (which are independent of $\alpha$, $x$ and $w$), such that
for
any two roots $\alpha,\beta$, every $x\in\R^d$ and any $w\in W$,
\begin{align}\label{C1C2}
|\alpha(x)|\leq C_1\,|x|\quad {\rm and}\quad |\partial_\alpha^x(\beta(w\,x))|\leq C_2.
\end{align}
\end{remark}
 
\subsection{Basic root subsystems}

\begin{definition}\label{basicsub}
A subsystem of a root system is a subset of the root system which is stable under the reflections with respect to the roots in the subset. 

We will say that $\Phi'$ is a root subsystem of $\Phi$ generated by the set $S\subseteq \Phi$ if it is the smallest subsystem of $\Phi$ containing $S$.

We will say that $\Phi'$ is a {\bf basic} root subsystem of $\Phi$ if it is generated by simple roots of $\Phi$.

 The {\bf rank} of $\Phi'$ is defined as the number of simple roots in $\Phi'$.

To simplify the exposition, we will automatically associate the following objects to a basic root subsystem $\Phi'$ of $\Phi$: let $W'$ be the group generated by $\sigma_{\alpha}$, $\alpha\in\Phi'$ and consider the polynomial $\pi'(x)=\prod_{\alpha\in\Phi'}\,\alpha(x)$.  We also denote 
$W_0=W\setminus W'$, $\Phi_0=\Phi\setminus \Phi'$ and let $\Phi_0^+$ stand for the positive roots in $\Phi_0$. When $\Phi'=\emptyset$, we will set $W'=\{id\}$ and therefore $\Phi_0=\Phi$ and $W_0=W\setminus\{id\}$.
\end{definition}

As an example, we present now the basic root subsystems in the $A_n$ case.

\begin{example}
If $\Phi'$ is a basic root subsystem of $A_n$ with simple roots $\alpha_i(x)=x_i-x_{i+1}$, $1\leq i\leq n$, then the set of simple roots in $\Phi'$ can be written as
\begin{align*}
\bigcup_{k=1}^M\,\{\alpha_{i_k},\alpha_{i_k+1},\dots,\alpha_{i_k+j_k-1}\}
\end{align*}
with $i_k+j_k<i_{k+1}$ whenever $1\leq k\leq M-1$. This means that
$
\Phi'\simeq\bigoplus_{k=1}^M\,A_{j_k}.
$
The rank of $\Phi'$ is $\sum_{k=1}^M\,j_k$.
\end{example}


\section{Estimates of the Poisson kernel in the Dunkl setting}\label{Dunkl setting}

In this section, we present results that are true in the general invariant Dunkl setting, without any restriction on the multiplicity function $k$.

\subsection{Framing bounds for the Dunkl-Poisson kernel}\label{framing}

Observe that, by \eqref{PoissonDunkl}, {for $x\in B$ and $y\in S$,}
\begin{align}
P_k(x,y)
=C\,V_k\left(\frac{1-|x|^2}{(|x-y|^2+2\,\langle x,y-\cdot\rangle)^{\kappa+d/2}}\right)(y).\label{F3}
\end{align}

From \eqref{F3} and \eqref{intertwiner}, we deduce the following bounds.

\begin{proposition}\label{ENC}
Let $w_{max}\in W$ be such that $|x-w_{max}y|=\max_{w\in W} |x-wy|$.
There exists $C>0$ such that for any $x\in B^+$, $y\in S^+$,
\begin{align}
\frac{C{(1-|x|^2)}}{|x-w_{max}y|^{d+2\,\kappa}}\leq P^W(x,y)\leq \frac{C{(1-|x|^2)}}{|x-y|^{d+2\,\kappa}},
\label{bounds}
\end{align}
for all $x\in B^+$ and $y\in S^+$.
\end{proposition}

\begin{proof}
The map $z\mapsto |x-y|^2+2\langle x,y-z\rangle$ is linear in $z\in C(y)$, the convex hull of $W\cdot y$. Therefore it attains its maximum and its minimum at points in $W\cdot y$. When $z=wy$, the denominator $|x-y|^2+2\langle x,y-z\rangle$ in \eqref{F3} equals
\begin{align}\label{useful}
|x-y|^2+2\langle x,y-w\,y\rangle=|x-wy|^2.
\end{align}

 The smallest value is $|x-y|^2$ and the largest is $|x-w_{max}y|^2$.
\end{proof}

\begin{remark} \label{generalDunkl}
Proposition \ref{ENC} also holds in the general Dunkl non invariant case. This is due to
\eqref{intertwiner} and \eqref{F3}. An analogous estimate of the Newton kernel was obtained in \cite[Proposition 6.3 (6.6)]{Gallardo2}.
\end{remark}

\begin{proposition}\label{smallgamma}
Let $D>0$. Then there exists $C>0$ independent of $D$ such that
\begin{align*}
C\,(1+4\,r\,MD\,(D+\max_{\alpha>0}\,|\alpha|))^{-d/2-\kappa}\leq \frac{P^W(x,y)}{\Omega(x,y)}\leq C\,(1+D)^{2\,\kappa}
\end{align*}
for all $x\in B^+$ and $y\in S^+$ with $\alpha(y)\leq D|x-y|$ for every $\alpha>0$.
\end{proposition}

\begin{proof}
Note that {by Lemma \ref{basic simple}(i),} $\alpha(x)=|\alpha|\,d(x,H_\alpha)\leq |\alpha|(|x-y|+d(y,H_\alpha))
=|\alpha||x-y|+\alpha(y)$ so $\alpha(x)\leq |x-y|\,\max_{\alpha>0}|\alpha|+D\,|x-y|=D'\,|x-y|$.
Using {\eqref{useful}, Proposition \ref{XWY} and Remark \ref{Up}}, 
\begin{align*}
|x-wy|^2&=|x-y|^2+2\langle x,y-wy\rangle = |x-y|^2+2\sum_{i=1}^r2\,a_i^w(y)\alpha_i(x)/|\alpha_i|^2\\
&\leq |x-y|^2+4\,r\,M\,\max_{\alpha>0}\,\alpha(y)\,\max_{\alpha>0}\,\alpha(x)
\leq (1+4\,r\,MD\,D')|x-y|^2.
\end{align*}
{
The result then follows from \eqref{bounds} in Proposition \ref{ENC}, noting that for all $\alpha>0$ we have $|x-y|\leq\phi_\alpha(x,y)\leq (1+D)\,|x-y|$.}
\end{proof}

\begin{lemma}\label{A1Rd}
Let $E=\{ (x,y)\in\a\times\a \ |\ \alpha(x)\,\alpha(y)>0 \}$. Consider
\begin{align*}
T_1(x,y)=\frac{\frac{1}{|x-y|^d}-\frac{1}{|x-\sigma_\alpha y|^d}}{\alpha(x)\,\alpha(y)}
\end{align*}
on $E$. Then, on $E$, the following estimate holds
\begin{align*}
T_1(x,y)\asymp\frac{1}{|x-y|^d\,|x-\sigma_\alpha y|^2}.
\end{align*}
\end{lemma}

\begin{proof}
Recall that, by Lemma \ref{basic simple}(ii), $|x-\sigma_\alpha y|^2=|x-y|^2+C\,\alpha(x)\,\alpha(y)$ where $C=4/|\alpha|^2$.
Using the formula $a^d-b^d=(a-b)\,\sum_{k=0}^{d-1}\,a^k\,b^{d-1-k}$,
and the fact 
that $|x-y|^d\leq |x-\sigma_\alpha y|^d$ on $E$,
we have
\begin{align*}
T_1(x,y)&=\frac{(|x-\sigma_\alpha y|^2)^d-(|x-y|^2)^d}{\alpha(x)\,\alpha(y)\,|x-y|^d\,|x-\sigma_\alpha y|^d\,(|x-y|^d+|x-\sigma_\alpha y|^d)}\\
&=C\,\frac{\sum_{k=0}^{d-1}\,|x-\sigma_\alpha y|^{2\,k}\,|x-y|^{2\,(d-1-k)}}{|x-y|^d\,|x-\sigma_\alpha y|^d\,(|x-y|^d+|x-\sigma_\alpha y|^d)}\\
&\asymp \frac{|x-\sigma_\alpha y|^{2\,(d-1)}}{|x-y|^d\,|x-\sigma_\alpha y|^{2\,d}}=\frac{1}{|x-y|^d\,|x-\sigma_\alpha y|^2}.
\end{align*}
\end{proof}

\section{{The conjecture for complex root systems}}\label{sec:complex}

\subsection{Structure of the proof of Conjecture \ref{TL}}

When the root system is complex, i.e. all the multiplicities $k(\alpha)=1$, by Proposition \ref{altern},
we have the alternating sum formula for the $W$-invariant Poisson-Dunkl kernel  at our disposal, see also \cite{PGPS}.
We have, for any $x\in B^+$, $y\in S^+$,
\begin{align}\label{alt}
\frac{P^W(x,y)}{1-|x|^2}=\frac1{|W|w_d}\frac{1}{\pi(x)\,\pi(y)}\,\left[\sum_{w\in W}\,\frac{\epsilon(w)}{|x-w\,y|^d}\right].
\end{align}
For a fixed basic root subsystem $\Phi'$, it is natural to decompose the alternating sum
\begin{align*}
\sum_{w\in W}\,\frac{\epsilon(w)}{|x-w\,y|^d}=
\sum_{w\in W'}\,\frac{\epsilon(w)}{|x-w\,y|^d}+
\sum_{w\in W_0}\,\frac{\epsilon(w)}{|x-w\,y|^d},
\end{align*}
and to apply an induction argument to the ``main term''
$\sum_{w\in W'}\frac{\epsilon(w)}{|x-w\,y|^d}$.
This will require a detailed analysis of the remainder term
$\sum_{w\in W_0}\,\frac{\epsilon(w)}{|x-w\,y|^d}$.
The choice of the basic root subsystem $\Phi'$
and further analysis of the main term and the remainder,
will be done by considering $(x,y)$ in the fixed subregions \eqref{AA} defined in Lemma \ref{reduc}.

\subsection{Subregions}

Let $N$ denote the maximal length of the positive roots in $\Phi$.  For example, if the root system is $A_n$ then $N=n$.
\begin{lemma}\label{reduc}
Consider a basic root subsystem $\Phi$ and fix $c>0$. Then, given $x$ and $y \in\overline{\a^+}$, it is possible to build a basic root subsystem $\Phi'$ of $\Phi$ and $c'$ with $c/N^r\leq c'\leq c$ where $r$ is the rank of $\Phi$
such that
\begin{align}
 c'\,\alpha(y)< |x-y|<c'\,\beta(y)\label{AA}
\end{align}
for every $\alpha\in \Phi'_+$ and every $\beta\in \Phi_0^+$ (this includes the possibility that $\Phi'=\emptyset$).
\end{lemma}

\begin{proof}

Let $c_0=\frac{c}{N^r}$ and let $S_0$ be the set of simple roots $\alpha_i$ such that
such that 
\begin{align}
c_0\,\alpha_i(y)< |x-y|.\label{chain}
\end{align}

    Let $\Phi'$ be generated by the simple roots such that \eqref{chain} holds (there are at most $r$ such simple roots).  If 
$c_0\,\alpha(y)< |x-y|$ for every $\alpha\in \Phi'$ then
$\Phi'$ satisfies the condition \eqref{AA} with $c'=c_0$ except perhaps $c_0\,\beta(y)=|x-y|$ for some $\beta\in \Phi_0^+$. In that case, we choose $0<\delta<c-c_0$ such that the inequalities $ (c_0+\delta)\,\alpha(y)< |x-y|$ hold for every $\alpha\in \Phi'_+$.  We then take $c'=c_0+\delta$.

Suppose $c_0\,\alpha(y)\geq |x-y|$ for some $\alpha\in \Phi'$. In that case since $\alpha$ is not one of the $\alpha_i$'s from \eqref{chain}, it is not simple and we 
have $c_0\,\alpha_j>|x-y|/N$ for some $j\in S_0$. We set $c_1=N\,c_0$ and repeat the process.

Let $S_1$ be the set of simple roots $\alpha_i$ such that
such that 
\begin{align}
c_1\,\alpha_i(y)< |x-y|.\label{chain2}
\end{align}

Let $\Phi'$ be generated by the simple roots such that \eqref{chain2} holds (there are at most $r-1$ such simple roots).  If 
$c_1\,\alpha(y)< |x-y|$ for every $\alpha\in \Phi'$ then
$\Phi'$ satisfies the condition \eqref{AA} with $c'=c_1$ (if $c_1\,\beta(y)=|x-y|$ for some $\beta\in\Phi_0$, we proceed as in the previous step).

Suppose $c_1\,\alpha(y)\geq |x-y|$ for some $\alpha\in \Phi'$. In that case, we 
have $c_1\,\alpha_j(y)>|x-y|/N$ for some $j\in S_1$. We set $c_2=N\,c_1$ and repeat the process.

Since we decrease the number of simple roots involved each time, the process must stop eventually either because we have achieved \eqref{AA} or because we ran out of simple roots.
\end{proof}

\begin{remark}
In Lemma \ref{reduc}, we do not exclude the possibility that $\Phi'$ might be empty (with $\Phi_0=\Phi$ and $W_0=W\setminus\{id\}$) or, at the other extreme, that $\Phi'=\Phi$ (with $\Phi_0=\emptyset$ and $W_0=\emptyset$).
\end{remark}

The subregions defined in Lemma \ref{reduc} will be denoted by
\begin{align*}
S_{\Phi',c}=\{(x,y)\in B^+\times S^+\colon c\,\alpha(y)< |x-y| < c\,\beta(y)
\ {\rm for\ all}\ \alpha\in \Phi'_+, \beta\in \Phi_0^+
\}
\end{align*}
The subregions $S_{\Phi',c}$ will play a crucial role in the proof of Conjecture \ref{TL}.

\subsection{Estimates of the remainder}
This subsection is devoted to the estimates of the remainder term $\sum_{w\in W_0}\,\frac{\epsilon(w)}{|x-w\,y|^d}$. We begin with lower estimates of the expressions $|x-w\,y|$ when
$w\in W_0$.

\begin{lemma}\label{blocks}
Let $\Phi'$ be a basic root subsystem of $\Phi$ with $\Phi'\not=\Phi$, $x$, $y\in\overline{\a^+}$ and $w_0\in W_0$. Then there exists a simple root $\alpha\in\Phi_0^+$ such that
$|x-w\,y|^2\geq K_0\,\alpha(x)\,\alpha(y)$
where $K_0=\max\,\{1/|\alpha|^2\colon \alpha\in\Phi_+\}$.
\end{lemma}

\begin{proof}
We use the decomposition 
\eqref{CL} of $y-wy$. Let
$\alpha_i$, $i=1,\ldots, s=| \Phi'_+|$ be the simple positive roots from $\Phi'$ and $\alpha_{s+j}$ the simple positive roots from $\Phi_0^+$. We have
\begin{align}\label{TwoSums}
y-w\,y=\sum_{i=1}^s \,2\,\frac{a_i^w(y)}{|\alpha_i|^2}\,A_{\alpha_i} + \sum_{j=1}^{r-s} \,2\,\frac{a_{s+j}^w(y)}{|\alpha_{s+j}|^2}\,A_{\alpha_{s+j}} 
\end{align}

In order to prove the lemma it is sufficient to show
that the second sum does not vanish since by Corollary \ref{alpha k in ak}, a nonzero function $a_{s+j}^w$ has a nonzero term
$n\, \alpha_{s+j}$.

Suppose by contradiction that the second sum vanishes.
Choose $x$ and $y$ with $\Phi_x=\Phi_y=\Phi'$.
Formula \eqref{TwoSums} implies that
\begin{align*}
\langle x,y-w\,y\rangle=\sum_{i=1}^s \,2\,\frac{a_i^w(y)}{|\alpha_i|^2}\,{\alpha_i(x)} =0, 
\end{align*}
since $\alpha_i(x)=0$. By Corollary \ref{KillMaxSameWall}, we get $w\in W_y=W'$.
\end{proof}

\begin{remark}
Lemma \ref{blocks} holds with the same proof if $\Phi'=\emptyset$, $\Phi_0=\Phi$ and $W_0=W\setminus\{id\}$.
\end{remark}

\begin{corollary}\label{A}
Let $\Phi'$ be a basic root subsystem of $\Phi$ and $x$, $y\in\bar\a^+$. Suppose $0<c\leq 1/(4\,C_1)$ where $C_1$ is as in Remark \ref{C} and
\begin{align*}
|x-y|\leq c\,\alpha(y)
\end{align*}
for all $\alpha \in\Phi_0^+$. Then if $w\in W_0$ and $K_0$ is as in Lemma \ref{blocks}, we have
\begin{align*}
\alpha(y)/2\leq \alpha(x)\leq 2\,\alpha(y)
\end{align*}
and
\begin{align}
\frac{|x-y|}{|x-w\,y|}\leq \frac{c}{\sqrt{K_0/2}}.\label{K0}
\end{align}
\end{corollary}

\begin{proof}
By Remark \ref{C}, if $\alpha\in \Phi_0^+$ then 
\begin{align*}
\alpha(x)\geq \alpha(y)-C_1\,|x-y|\geq \alpha(y)-C_1\,c\,\alpha(y)
=(1-C_1\,c)\,\alpha(y)\geq \alpha(y)/2.
\end{align*}

On the other hand, $\alpha(y)\geq 
\alpha(x)-C_1\,|x-y|\geq \alpha(x)-C_1\,c\,\alpha(y)
\geq \alpha(x)-2\,C_1\,c\,\alpha(x)
=(1-2\,C_1\,c)\,\alpha(x)\geq \alpha(x)/2$. The rest follows from Lemma \ref{blocks}.
\end{proof}


\begin{lemma}\label{side} (Conjecture \ref{TL} in the case $\Phi'=\emptyset$)
If $0<c\leq \min\{\frac{1}{4\,C_1},\frac{\sqrt{K_0/2}}{\sqrt[d]{2\,|W|}}\}$ where $C_1$ is as in Remark \ref{C}, $K_0$ is as in Lemma \ref{blocks} and $|x-y|\leq c\,\alpha(y)$ for every $\alpha\in \Phi$ then the Conjecture \ref{TL} holds.
\end{lemma}

\begin{proof} 
In that case, $\Phi'=\emptyset$, $\Phi_0=\Phi$, $W_0=W\setminus\{id\}$ and by Corollary \ref{A}, for every $w\in W_0$, 
equation \eqref{K0} holds and
\begin{align*}
\lefteqn{\frac{1-|x|^2}{w_d\,\pi(x)\,\pi(y)} \,\frac{1}{|x-y|^{d}}\geq P^W(x,y)
\geq\frac{1-|x|^2}{|W|\,w_d\,\pi(x)\,\pi(y)} \,\left[\frac{1}{|x-y|^{d}}-\sum_{w\not=id}\,\frac{1}{|x-w\,y|^d}
\right]}\\
&\geq\frac{1-|x|^2}{|W|\,w_d\,\pi(x)\,\pi(y)} \,\left[\frac{1}{|x-y|^{d}}-\frac{c^d}{(K_0/2)^{d/2}}\,\sum_{w\not=id}\,\frac{1}{|x-y|^d}
\right]\\
&\geq(1-c^d\,|W|/(K_0/2)^{d/2})\,\frac{1-|x|^2}{|W|\,w_d\,\pi(x)\,\pi(y)} \,\frac{1}{|x-y|^{d}}
\geq\frac{1}{2}\,\frac{1-|x|^2}{|W|\,w_d\,\pi(x)\,\pi(y)} \,\frac{1}{|x-y|^{d}}.
\end{align*}
Hence, the conjecture holds in this case since $|x-\sigma_\alpha\,y|^2=|x-y|^2+C\,\alpha(x)\,\alpha(y)\asymp\alpha(x)\,\alpha(y)$ by the hypothesis $|x-y|\leq c\,\alpha(y)$ and Corollary \ref{A}.
\end{proof}
{
In the second part of this subsection devoted to estimates of the remainder $\sum_{w\in W_0}\,\frac{\epsilon(w)}{|x-w\,y|^d}$, we
 consider all root subsystems $\Phi''\subseteq\Phi'$
 and we deal with the function $R_{\Phi''}(x,y)$, defined in the following definition. This will be essential for the upper estimates of the remainder on the subregions $S_{\Phi',c}$.
}

\begin{definition}
Let $\Phi'$ be a basic root subsystem of $\Phi$, $\Phi''$ a subsystem of $\Phi'$ and $I=|\Phi''|$. We define, when $(x,y)\in D_0=\{(x,y) {\in\a\times\a} \ |\ \prod_{w\in W_0}|x-w\,y|\not= 0 \}$,
\begin{align}\label{Nxy}
R_{\Phi''}(x,y)=|x-y|^{d+2\,I}\,\sum_{w\in W_0}\,\frac{\epsilon(w)}{|x-w\,y|^d}
\end{align}
\end{definition}

\begin{remark} 
When the subsystem $\Phi''$ is fixed and it does not lead to misunderstanding, in order to simplify the notation, we will write 
\begin{align*}
R_{\Phi''}(x,y)=R(x,y)
\end{align*}
even though $R(x,y)$ depends on $\Phi''$
via the factor $|x-y|^{d+2\,I}=|x-y|^{d+2\,|\Phi''|}$.
\end{remark}

\begin{lemma}\label{Bound} 
Let $\Phi'$ be a basic root subsystem of $\Phi$. Let $\Phi''\not=\emptyset$ be a basic root subsystem of $\Phi'$.
Let $R(x,y)$ be defined by \eqref{Nxy}.
{Denote} $\partial^y(\pi'')=\prod_{\alpha\in \Phi''}\,\partial^y_\alpha$. Then there exists a constant $K$ independent of $x$ and $y$ such that for all $(x,y)\in D_0$
\begin{align*}
\left|\partial^x(\pi'')\,\partial^y(\pi'')\,R(x,y)\right|\ \le \ K\,\max_{0\leq i \leq 2\,I, w\in W_0}\,\left\{\frac{|x-y|^{d+i}}{|x-w\,y|^{d+i}}\right\}.
\end{align*}
\end{lemma}

\begin{proof}
First note that
\begin{align}\label{der}
\partial^x_\alpha\,|x-y|^P =P\,\alpha(x-y)\,|x-y|^{P-2}
\end{align}
and
\begin{align}\label{first der}
\partial^x_\alpha\,\left[|x-y|^{d+2\,I}\,|x-w\,y|^{-d}\right]
&=(d+2\,I)\,\alpha(x-y)\,|x-y|^{d+2\,I-2}\,|x-w\,y|^{-d}\nonumber \\
&+(-d)\,\alpha(x-w\,y)|x-y|^{d+2\,I}|x-w\,y|^{-d-2}.
\end{align}

In the computations below, it is helpful to remember that $|x-w\,y|=|w^{-1}\,x-y|$.
Using formulas \eqref{der} and \eqref{first der}, we obtain for each $w\in W_0$,
\begin{align*}
\lefteqn{\partial^y_\alpha\,\partial^x_\alpha\,|x-y|^{d+2\,I}\,|x-w\,y|^{-d}}\\
&=-2\,(d+2\,I)\,|x-y|^{d+2\,I-2}\,|x-w\,y|^{-d}\\
&+{(d+2\,I-2)}\,(d+2\,I)\,\alpha(x-y)\,\alpha(y-x)\,|x-y|^{d+2\,I-4}\,|x-w\,y|^{-d}\\
&-d\,(d+2\,I)\,\alpha(x-y)\,\alpha(y-w^{-1}\,x)\,|x-y|^{d+2\,I-2}\,|x-w\,y|^{-d-2}\\
&-{(d+2\,I-2)}\,(d+2\,I)\,\alpha(y-x)\,\alpha(x-w\,y)\,|x-y|^{d +2\,I-2}\,|x-w\,y|^{-d-2}\\
&-{(d+2\,I-2)}\,(d+2\,I)\,\alpha(y-x)\,\alpha(x-w\,y)\,|x-y|^{d+2\,I-2}\,|x-w\,y|^{-d-2}\\
&+d\,\partial_\alpha^y(\alpha(w\,y))\,	|x-y|^{d+2\,I}\,|x-w\,y|^{-d-2}\\
&+d\,(d+2)\,\alpha(x-w\,y)\,\alpha(y-w^{-1}\,x)\,|x-y|^{d+2\,I}\,|x-w\,y|^{-d-4}.
\end{align*}

Observe that 
 after each application of an operator
$\partial^y_\alpha\,\partial^x_\alpha$, where $\alpha$
is a root from $\Phi''$, we get {a sum of} terms 
which are products of constants times terms in $\alpha(x-y)$, $\alpha(w^{-1}\,x-y)$, $\alpha(x-w\,y)$, $|x-y|^{d+2\,I-2\,{k}}$, 
and $|x-w\,y|^{-d-{2l}}$. {We used} the fact that {a directional} derivative of a linear term such as $\alpha(x-y)$ {is} a constant.

By applying {for the first time} $\partial^y_\alpha\,\partial^x_\alpha$, the degree
{of homogeneity of the function $|x-y|^{d+2\,I}\,|x-w\,y|^{-d}$} 
went from $2\,I$ 
to $2\,I-2$.

Each new {application of an operator $\partial^y_\alpha\,\partial^x_\alpha$}
decreases the degree by 2 until we get {the} degree of 0 once we have applied $\partial^x({\pi''})\,\partial^y({\pi''})$.

{We use the estimates \eqref{C1C2} from Remark \ref{C}.}
From the fact that $|\alpha(x-y)|\leq C_1\,|x-y|$, $|\alpha(x-w\,y)|\leq C_1\,|x-w\,y|$, $|\alpha(w^{-1}\,x-y)|\leq C_1\,|w^{-1}\,x-y|=C_1\,|x-w\,y|$, we
note that
$\left|\partial^x({\pi''})\,\partial^y({\pi''})\,R(x,y)\right|$ is bounded by
a finite sum of terms of the form
$C\, |x-y|^{d+2\,I-\,{j}}\,|x-w\,y|^{-d-{(2I-\,j)}}$,
or, equivalently,
\begin{align*}
C\, \frac{|x-y|^{d+i}}{|x-w\,y|^{d+i}}
\end{align*}
where $0\leq i\leq 2\,I$. The rest follows.
\end{proof}

\begin{lemma}
Let $\Phi'$ be a basic root subsystem and let $\alpha_1$, \dots,$\alpha_M$ be distinct positive roots of $\Phi'$ with $M\leq |\Phi'|$ and let $\pi'(x)=\prod_{\alpha\in\Phi'_+}\,\langle\alpha,x\rangle$. Then
\begin{align}
\partial_{\alpha_M}\,\cdots\,\partial_{\alpha_1}\,\pi'(x)
=\sum_{\beta_1,\dots,\beta_M}\,\prod_{i=1}^M\,\langle\beta_i,\alpha_i\rangle\,\prod_{\alpha\in\Phi'_+\setminus\{\beta_1,\dots,\beta_M\}}\,\langle\alpha,x\rangle
\label{dpi}
\end{align}
where the sum is taken over distinct positive roots and the order of the $\beta_i$'s is important.
\end{lemma}

\begin{proof}
First note that $\partial_{\alpha_k}\,\langle\alpha,x\rangle=\langle\alpha,\alpha_k\rangle$. The result then follows easily using induction on $M$ and Leibnitz rule.
\end{proof}

\begin{corollary}\label{intermediary}
Let $\Phi'$ be a basic root subsystem and let $\alpha_1$, \dots,$\alpha_M$ be distinct positive roots of $\Phi'$ with $M<|\Phi'_+|$ and let $\pi'(x)=\prod_{\alpha\in\Phi'_+}\,\langle\alpha,x\rangle$. Then for any positive root $\alpha_{M+1}$ other than one of the $\alpha_k$, $1\leq k\leq M$, we have
\begin{align*}
\partial_{\alpha_M}\,\cdots\,\partial_{\alpha_1}\,\pi'(x)=0
\end{align*}
on the set $\{x\colon\alpha_k(x)=0,~k=1,\dots,M+1\}$.
\end{corollary}

\begin{proof}
Setting $\alpha_i(x)=0$ for $i=1$, \dots, $M$ in formula \eqref{dpi} gives 
\begin{align*}
\partial_{\alpha_M}\,\cdots\,\partial_{\alpha_1}\,\pi'(x)
=\left[\sum_{\sigma\in S_M}\,\prod_{i=1}^M\,\langle\alpha_{\sigma(i)},\alpha_i\rangle\right]\,\prod_{\alpha\in\Phi'_+\setminus\{\alpha_1,\dots,\alpha_M\}}\,\langle\alpha,x\rangle.
\end{align*}
The rest is straightforward: the last product cancels on $S$ because it contains the term $ \langle\alpha_{M+1},x\rangle$.
\end{proof}

\begin{lemma}\label{nonzero}
Let $\Phi'$ be a basic root subsystem. Then $\partial(\pi')\,\pi'(x)$ is a nonzero constant.
\end{lemma}

\begin{remark}
The exact value of $\partial(\pi')\,\pi'(x)$ is computed in \cite[Chap.{} II, Cor.{} 5.36]{Helgason2}. We provide here an elementary proof that $\partial(\pi')\,\pi'(x)\not=0$
\end{remark}

\begin{proof}
It is clear that $\partial(\pi')\,\pi'(x)$ is a constant given that the degree of the operator $\partial(\pi')$ is equal to the degree of $\pi'(x)$. Now, let 
\begin{align*}
S_\lambda(x)=
\frac{\sum_{w\in W}\,\epsilon(w)\,e^{\langle \lambda,w\,x\rangle}}{\pi'(x)}=\frac{\sum_{w\in W}\,\epsilon(w)\,e^{\langle w^{-1} \lambda,x\rangle}}{\pi'(x)}
\end{align*}
where $W$ is generated by the reflections $\sigma_\alpha$, $\alpha\in\Phi'$. Now, $S_\lambda$ is an analytic function (the numerator is $W$ skew-symmetric and analytic; in fact $S_\lambda$ is a constant multiple of $\pi(\lambda)$ times a spherical function). If we compute $\partial(\pi')\,(\pi'(x)\,S_\lambda(x))$ and set $x=0$, we get
\begin{align*}
\left.\partial(\pi')\,(\pi'(x)\,S_\lambda(x))\right|_{x=0}=\sum_{w\in w}\,\epsilon(w)\,\pi'(w^{-1}\,\lambda)=|W|.
\end{align*}

The left hand side is equal to $\partial(\pi')\,(\pi'(x))\,S(0)$ since all terms involving derivatives of $S_\lambda$ are eliminated by one of the roots (a consequence of Corollary \ref{intermediary}). We can therefore conclude that $\partial(\pi')\,(\pi'(x))\not=0$.
\end{proof}

\begin{proposition}\label{Hospital+}
Let $\Phi'$ be a basic root subsystem, $\Phi''$ a subsystem of $\Phi'$ and let $P_{\Phi''}=\bigcup_{\alpha\in \Phi''}\,H_\alpha$ and $\pi''(x)=\prod_{\alpha\in\Phi''}\,\alpha(x)$. 
Suppose $R(x,y)$ is defined as in \eqref{Nxy}.
Note that $R(x,y)=Q(x,y)\,\pi''(x)\,\pi''(y)$ where $Q$ is analytic since $R$ is analytic and skew-symmetric with respect to $W''$ (the Weyl subgroup generated by the reflections $\sigma_\alpha$, $\alpha\in\Phi''$) in both variables.
Then there exists a constant $K$ independent of $x$ and $y$ such that
\begin{align*}
\sup_{x,y\in \a^+\setminus P_{\Phi''}}\,\frac{|R(x,y)|}{\pi''(x)\,\pi''(y)}
&\le 
\sup_{x,y\in \a^+\setminus P_{\Phi''}}\frac{|\partial^x(\pi'')\,\partial^y(\pi'')\, R(x,y)|}
{|\partial^x(\pi'')\,\partial^y(\pi'')\,(\pi''(x)\,\pi''(y)) |}\\
&=
K\,\sup_{x,y\in \a^+\setminus P_{\Phi''}}|\partial^x(\pi'')\,\partial^y(\pi'')\, R(x,y)|.
\end{align*}
\end{proposition}

\begin{proof}
We use the Cauchy Mean Value Theorem several times as in the proof of the multivariate de l'Hospital rule proposed by Lawlor \cite{Lawlor}. More precisely, 
if {$x$, $y\not\in H_\alpha$}, then, denoting by $x_\alpha$ and $y_\alpha$ the projections of $x$ and $y$ on $H_\alpha$,
\begin{align*}
\frac{R(x,y)}{\alpha(x)\,\alpha(y)}=\frac1{\alpha(y)}\frac{R(x,y)-R(x_\alpha,y)}{\alpha(x)-\alpha(x_\alpha)}=
\frac1{\alpha(y)}\frac{\partial_\alpha^xR(x',y)}{\partial_\alpha^x\alpha(x')}
\end{align*}
for an intermediate point $x'\in \a^+\setminus H_\alpha$.
The same reasoning allows us to repeat the argument with $\partial_\alpha^y$:
\begin{align*}
\frac{\partial_\alpha^x\,R(x',y)}{\alpha(y)}
=\frac{\partial_\alpha^x\,R(x',y)-\partial_\alpha^x\,R(x',y_\alpha)}{\alpha(y)-\alpha(y_\alpha)}
=\frac{\partial_\alpha^y\,\partial_\alpha^x\,R(x',y')}{\partial_\alpha^y\alpha(y')}.
\end{align*}

This was done for a single $\alpha$ but we can repeat the process for all the roots $\alpha$, $\alpha\in {\Phi''}$. This is made possible by Corollary \ref{intermediary}. Note that $\partial^x(\pi'')\,\partial^y(\pi'')\,(\pi''(x)\,\pi''(y))\not=0$ by Lemma \ref{nonzero}.
\end{proof}

\begin{proposition}\label{N}
Let $\Phi'$ be a basic root subsystem of $\Phi$ and $\Phi''$ a subsystem of $\Phi'_+$. 
Suppose $0<c\leq \min\{1,1/(4\,C_1)\}$ where $C_1$ is as in Remark \ref{C} and let
\begin{align}
S_{\Phi',c}=\{(x,y)\in B^+\times S^+\colon c\,\alpha(y)< |x-y| < c\,\beta(y)
\ {\rm for\ all}\ \alpha\in \Phi'_+, \beta\in \Phi_0^+
\}.\label{S}
\end{align}
{Suppose $R(x,y)$ is defined as in \eqref{Nxy}.}
Then there exists a constant $K$ independent of $x$ and $y$ such that
\begin{align*}
\sup_{x,y\in S_{\Phi',c}}\,\frac{|R(x,y)|}{\pi''(x)\,\pi''(y)}
\leq K\,c^d.
\end{align*}
\end{proposition}

\begin{proof}
This result follows directly from Corollary \ref{A}, Lemma \ref{Bound} and Proposition \ref{Hospital+}.
\end{proof}

\subsection{Proof of Conjecture \ref{TL}}
\begin{theorem}\label{TLProof}
Conjecture \ref{TL} is valid for all complex root systems.
\end{theorem}

\begin{proof}
We use induction on the rank $r$ of the root system. The result is true in rank 1 by Lemma \ref{A1Rd}.

Assume that it is true for ranks 1,\ldots, $r-1$, $r\geq 2$. Let $\Phi$ be a rank $r$ complex root system. The induction hypothesis means that there exist constants $K_1$ and $K_2$ such that for 
 $x\in B^+,y\in S^+$, we have
\begin{align}
K_1\leq |x-y|^d\,\prod_{\alpha\in \Phi'_+}\phi_\alpha(x,y)^{2}\,\frac{\sum_{w\in W'}\,\frac{\epsilon(w)}{|x-w\,y|^d}}{\pi'(x)\,\pi'(y)}\leq K_2\label{B1}
\end{align}
for every basic root subsystem $\Phi'$ of $\Phi$ which is of lower rank. Since the number of such subsystems is finite, we 
may assume that $K_1$ and $K_2$ are independent from the choice of $\Phi'$ (depend only on $\Phi$).

Fix $c$ such that
\begin{align}\label{cond c1}
0<c\leq \min\{1,\frac{1}{4\,C_1},\frac{\sqrt{K_0/2}}{\sqrt[d]{2\,|W|}}\}
\end{align}
where $C_1$ is as in Remark \ref{C}, $K_0$ is as in Lemma \ref{blocks} and such that
\begin{align}\label{cond c2}
(1+2\,N\,C_1)^{2M}\,(2+C_1^{-1})^{2M}\,K\,c^d\leq K_1/2
\end{align}
where $K$ is as in Proposition \ref{Hospital+} (we consider all the basic root subsystems $\Phi''\subseteq \Phi'$ for all possible basic root subsystems $\Phi' $ of $\Phi$ and take the largest corresponding value of $K$). Recall that $N$ denotes the maximal length of positive roots in $\Phi$
and $M$ is the number of all positive roots in $\Phi$.
By Lemma \ref{reduc}, if $S_{\Phi',c'}$ is as in \eqref{S}, we have
\begin{align*}
B^+\times S^+=\bigcup_{\Phi',c'}\,S_{\Phi',c'}
\end{align*}
where the union is taken over all basic root subsystems $\Phi'$ of $\Phi$ and $c'$ with $c/N^r\leq c'\leq c$.
Since the number of such basic root subsystems is finite, we only have to prove the Conjecture \ref{TL} for a specific choice of $\Phi'$ where $x$, $y$ satisfy
\begin{align}
c'\,\alpha(y)< |x-y|< c'\,\beta(y)\label{this}
\end{align}
for $\alpha\in\Phi'_+$,$\beta\in \Phi_0^+$
and any $c'$ such that $c/N^r\leq c'\leq c$.

Recall that $c$ is fixed. We ensure that the constants in the estimates of the Conjecture \ref{TL} depend only on $c$, when $\Phi'$ is fixed
and $c'$ varies in the segment $[c/N^r,c]$.

We examine three possible cases:
\medskip

\noindent\underline{$\Phi'=\emptyset$}: in that case, Lemma \ref{side} allows us to conclude. 
\medskip

\noindent\underline{$\Phi'=\Phi$}: in that case, the result follows from \eqref{this} and from Proposition \ref{smallgamma} (if we refer to the notation of Proposition 
\ref{smallgamma}, 
$D=1/c'$ and the constants depend on $D$; this is where $c/N^r\leq c'\leq c $ comes into play).
\medskip

\noindent\underline{$1\leq \hbox{rank}(\Phi')< \hbox{rank}(\Phi)$}:
We have $\phi_\beta(x,y)=\beta(y)$ and $\beta(x)\,\beta(y)/2\leq \phi_\beta^2(x,y)\leq 2\,\beta(x)\,\beta(y)$ by Corollary \ref{A}.
Hence, 
\begin{align}\label{mixed}
2^{-M}\,\prod_{\beta\in\Phi_0^+}\,\beta(x)\,\beta(y)\leq \prod_{\beta\in\Phi_0^+}\,\phi_\beta(x,y)^2
\leq 2^{M}\,\prod_{\beta\in\Phi_0^+}\,\beta(x)\,\beta(y).
\end{align} 
We used the fact that $|\Phi_0^+|\leq |\Phi_+|\leq M$.

For a basic subsystem $\Phi''\subseteq \Phi'$, let 
\begin{align*}
X_{\Phi''}&=\{(x,y)\in S_{\Phi',c'}\colon \hbox{$\alpha(y)\leq 2 N\,C_1\,|x-y|$ for $\alpha\in \Phi''_+$
}\\&\qquad\qquad\qquad\qquad\hbox{ 
and $\alpha(y)>2\,C_1\,|x-y|$ for $\alpha\in \Phi'_+\setminus\Phi''$}\}
\end{align*}
and note that the union of all $X_{\Phi''}$ for $\Phi''\subseteq \Phi'_+$ gives $S_{\Phi',c'}$.
Indeed, for given $x$ and $y$, $\Phi''$ is generated by the simple roots $\alpha_i$ such that
$\alpha_i(y)\leq 2 \,C_1\,|x-y|$ (it may happen that $\Phi''=\emptyset$ in which case 
$X_{\Phi''}=\{(x,y)\in S_{\Phi',c'}\colon \alpha(y)>2\,C_1\,|x-y|~ \hbox{for $\alpha\in \Phi'_+$}\}$).

If $(x,y)\in X_{\Phi''}$ and $\alpha\in \Phi''$ then
\begin{align}\label{PhiBis1}
\phi_\alpha(x,y)^2\leq (1+2\,C_1)^2\,|x-y|^2
\end{align} 
If $(x,y)\in X_{\Phi''}$ and $\alpha\in \Phi'\setminus\Phi''$ then $\alpha(x)\geq \alpha(y)-C_1\,|x-y|> \alpha(y)/2$. Hence, $|x-y| < (2\,C_1)^{-1}\,\alpha(y)\leq 2\,(2\,C_1)^{-1}\,\alpha(x)$.
Therefore,
\begin{align}\label{PhiBis2}
\phi_\alpha(x,y)^2\leq (2+C_1^{-1})\,(1+(2\,C_1)^{-1})\,\alpha(x)\,\alpha(y)\leq (2+C_1^{-1})^2\,\alpha(x)\,\alpha(y).
\end{align}

Let
\begin{align*}
A&=\frac{|W|w_d P^W(x,y)}{\Omega(x,y)}=|x-y|^d\,\prod_{\gamma\in \Phi_+}\,\phi_\gamma(x,y)^{2}\,\frac{\sum_{w\in W}\,\frac{\epsilon(w)}{|x-w\,y|^d}}{\pi(x)\,\pi(y)}\\
&=\frac{\prod_{\beta\in \Phi_0^+}\,\phi_\beta(x,y)^{2}}{\prod_{\beta\in \Phi_0^+}\,\beta(x)\,\beta(y)}
\,\left[|x-y|^d\,\prod_{\alpha\in \Phi'_+}\,\phi_\alpha(x,y)^{2}\,\frac{\sum_{w\in W'}\,\frac{\epsilon(w)}{|x-w\,y|^d}}{\pi'(x)\,\pi'(y)}
\right.\\&\qquad\left.
+\underbrace{|x-y|^d\,\prod_{\alpha\in \Phi'_+}\,\phi_\alpha(x,y)^{2}\,\frac{
\sum_{w\in W_0}\,\frac{\epsilon(w)}{|x-w\,y|^d}
}
{\pi'(x)\,\pi'(y)}}_{\widetilde A}\right].
\end{align*}

Now, using \eqref{B1} and \eqref{mixed}, we get
\begin{align}\label{AB}
2^{-M}\,[K_1-|\widetilde A|]\leq A\leq 2^{M}\,[K_2+|\widetilde A|].
\end{align}

Let us fix $\Phi''\subset \Phi'$. For $(x,y)\in X_{\Phi''}$, 
using
inequalities \eqref{PhiBis1} and \eqref{PhiBis2}
and next the condition
\eqref{cond c2} on $c$ and Proposition 
\ref{N},
we obtain

\begin{align*}
|\widetilde A|&\leq
(1+2\,NC_1)^{2\,|\Phi''_+|}\,(2+C_1^{-1})^{2\,|\Phi'_+\setminus\Phi''_+|}\,
\left|\frac{R(x,y)}{\pi''(x)\,\pi''(y)}\right|\\
&\leq (1+2\,NC_1)^{M}\,(2+C_1^{-1})^{M}\,K\,c^d\leq K_1/2.
\end{align*}
This is enough to conclude from \eqref{AB} that
$A\asymp 1$.

\end{proof}

\section{Newton's kernel}\label{sec:Newton}

\begin{theorem}\label{NewtonProof}
Conjecture \ref{TLN} is valid for all complex root systems.
\end{theorem}

\begin{proof}
The proof follows the same pattern as the proof  of Theorem \ref{TLProof}
for the Poisson kernel with suitable adjustments.  Given that, unlike in the case of the Poisson kernel, the domain of the Newton kernel is unbounded, so it is worthwhile to review the ``road-map'' of the proof.

We first introduce the concept of the basic subalgebra generated by simple roots in Definition \ref{basic simple}.  The crucial step is then given in Lemma \ref{reduc} where it is shown that given a  Lie algebra $\Phi$ and $c>0$,  a basic subalgebra $\Phi'$ can be found that $ c'\,\alpha(y)< |x-y|<c'\,\beta(y)$ for $\alpha\in\Phi'$ and $\beta\in\Phi\setminus\Phi'$ with $c'$ in an interval bounded above and below by
universal multiples of $c$, namely $c'\in[c/N^r,c]$.

We start by showing that if $|x-y|\leq c\,\alpha(y)$ for every positive root $\alpha$ then Conjecture \ref{TLN} holds provided $c$ is chosen small enough.  The proof is the same as the proof of Proposition \ref{smallgamma} and it only depends on the formula \eqref{NewtonDunkl} analogous
to \eqref{PoissonDunkl}. This proof is valid in the general Dunkl case.

In the other extreme case, when $ c\,\alpha(y)\leq |x-y|$ for every positive root $\alpha$, the Conjecture \ref{TL} in the complex case holds provided that $c$ is small enough, by Lemma \ref{side}.  Once more, the same proof holds for the Newton kernel in the case $d\geq 3$.

It is important to note that the rank 1 case for $d\geq3$ which is handled by Lemma \ref{A1Rd} in the Poisson kernel case, follows practically the same proof in the case of the Newton kernel.

 The rest of the proof of Conjecture \ref{TLN}  follows the same line as that of Theorem \ref{TLProof}.  Indeed, as for the Poisson kernel, we define the subregions of $\overline{\a^+}\times\overline{\a^+}$ in the case of the Newton 
 kernel by
\begin{align*}
S_{\Phi',c}=\{(x,y)\in \overline{\a^+}\times\overline{\a^+} \colon c\,\alpha(y)< |x-y| < c\,\beta(y)
\ {\rm for\ all}\ \alpha\in \Phi'_+, \beta\in \Phi_0^+
\}.
\end{align*}

For a fixed value of $c$, the pair $x$, $y$ will fall in one of the regions $S_{\Phi',c'}$ where $c'$ is in a range bounded above and below by universal multiples of $c$.  Then there are  three cases:
$\Phi'=\Phi$, $\Phi'=\emptyset$ and the ``in between case''.  The first two situations are covered by the arguments given above.  For the ``in between case'', we proceed as in the case of the Poisson kernel and write 
\begin{align*}
A&=\frac{|W|\,(2-d)\,w_d\,N^W(x,y)}{\tilde{\Omega}(x,y)}=|x-y|^{d-2}\,\prod_{\gamma\in \Phi_+}\,\phi_\gamma(x,y)^{2}\,\frac{\sum_{w\in W}\,\frac{\epsilon(w)}{|x-w\,y|^{d-2}}}{\pi(x)\,\pi(y)}\\
&=\frac{\prod_{\beta\in \Phi_0^+}\,\phi_\beta(x,y)^{2}}{\prod_{\beta\in \Phi_0^+}\,\beta(x)\,\beta(y)}
\,\left[|x-y|^{d-2}\,\prod_{\alpha\in \Phi'_+}\,\phi_\alpha(x,y)^{2}\,\frac{\sum_{w\in W'}\,\frac{\epsilon(w)}{|x-w\,y|^{d-2}}}{\pi'(x)\,\pi'(y)}
\right.\\&\qquad\left.
+\underbrace{|x-y|^{d-2}\,\prod_{\alpha\in \Phi'_+}\,\phi_\alpha(x,y)^{2}\,\frac{\sum_{w\in W_0}\,\frac{\epsilon(w)}{|x-w\,y|^{d-2}}}
{\pi'(x)\,\pi'(y)}}_{\widetilde A}\right].
\end{align*}
where $\tilde{\Omega}(x,y)$ refers to the conjectured upper and lower bound for the Newton kernel (as per Conjecture \ref{TLN}).
The end of the proof is the same as in the proof of Theorem \ref{TLProof} and uses, among others, induction on the rank of a subsystem.
\end{proof}

\subsection{Invariant Newton kernel on $\R^2$}
Theorem \ref{main k1} does not cover the case $d=2$. Let us discuss the case of $A_1$ and $A_2$ root systems in $\R^2$. 

We use the normalization $\|\alpha_i\|^2=2$. It will be convenient to use
the functions 
\begin{align}
\psi_{\gamma_1,\gamma_2}(x,y)=2\,\gamma_1(x)\,\gamma_2(y)/|x-y|^2.\label{psi}
\end{align}

\begin{proposition}\label{A1m2}
In the case $A_1$ in $R^2$, we have
\begin{align*}
N^W(x,y)=-\frac{1}{2\,\pi}
\,\frac{1}{|x-y|^2}
\,\frac{\ln(1+\psi_{\alpha,\alpha}(x,y))}{\psi_{\alpha,\alpha}(x,y)}.
\end{align*}
In the case $A_2$ in $R^2$, we have 
\begin{align*} 
N^W(x,y)=-\frac{2}{\,\pi}\,\frac{1}{|x-y|^6} L(x,y),
\end{align*}
where the logarithmic correction factor is given by
\begin{align*}
L(x,y)=
{\ln\frac{(1+\psi_{\alpha,\alpha})\,(1+\psi_{\beta,\beta})(1+\psi_{\alpha+\beta,\alpha+\beta})}{(1+\psi_{\alpha,\alpha}+\psi_{\beta,\beta}+\psi_{\alpha,\beta})\,(1+\psi_{\alpha,\alpha}+\psi_{\beta,\beta}+\psi_{\beta,\alpha})}}.
\end{align*}
\end{proposition}
\begin{proof}
In rank 1 case with $d=2$, we have
\begin{align*}
N^W(x,y)&=\frac{1}{2\,\pi}\,\frac{\ln|x-y|-\ln|x-\sigma\,y|}
{\alpha(x)\,\alpha(y)}
=-\frac{1}{4\,\pi}\,\frac{\ln\frac{|x-\sigma\,y|^2}{|x-y|^2}}{\alpha(x)\,\alpha(y)}\\
&=-\frac{1}{2\,\pi}
\,\frac{1}{|x-y|^2}
\,\frac{\ln(1+\psi_{\alpha,\alpha}(x,y))}{\psi_{\alpha,\alpha}(x,y)}.
\end{align*}

In the $A_2$ case,
\begin{align*}
N^W(x,y)&=\frac{1}{2\,\pi}\,\sum_{w\in W}\,\epsilon(w)\,\frac{\ln|x-w\,y|}
{\pi(x)\,\pi(y)}
=\frac{1}{2\,\pi}\,\frac{\ln\frac{|x-y|\,|x-\sigma_\alpha\,\sigma_\beta\,y|\,|x-\sigma_\beta\,\sigma_\alpha\,y|}{|x-\sigma_\alpha\,y|\,|x-\sigma_\beta\,y|\,|x-\sigma_{\alpha+\beta}|}}{\pi(x)\,\pi(y)}\\
&=-\frac{1}{4\,\pi}\,\frac{\ln\frac{|x-\sigma_\alpha\,y|^2\,|x-\sigma_\beta\,y|^2\,|x-\sigma_{\alpha+\beta}|^2}{|x-y|^2\,|x-\sigma_\alpha\,\sigma_\beta\,y|^2\,|x-\sigma_\beta\,\sigma_\alpha\,y|^2}}{\pi(x)\,\pi(y)}\\
&=-\frac{C^3}{4\,\pi}\,\frac{1}{|x-y|^6}
\frac{L(x,y)}{\prod_{\gamma>0}\,\psi_{\gamma,\gamma}(x,y)}
\end{align*}
(using Lemma \ref{basic simple} repeatedly).
\end{proof}

It is noteworthy that, in the $A_1$ case, the behaviour of $N^W(x,y)$ for $\psi_{\alpha,\alpha}(x,y)\leq1$ is comparable with $-\frac{C}{4\,\pi}\,\frac{1}{|x-y|^2}$, and the behaviour for $\psi_{\alpha,\alpha}(x,y)\geq 1$ 
is comparable with $-\frac{C}{4\,\pi}\,\frac{1}{\alpha(x)\alpha(y)}\,\ln \psi_{\alpha,\alpha}(x,y)$.

In the $A_2$ case, the techniques from Section \ref{framing}
imply that for $\alpha(x),\beta(x)\le c|x-y|^2$ one has
$N^W(x,y)\asymp -\frac{1}{|x-y|^6}.$

Note that in the above cases, the estimates
\begin{align*}
N^W(x,y)\asymp \frac{N^{\R^2}(x,y)}{\prod_{\alpha > 0 \ }|x-\sigma_\alpha y|^{2}}
\end{align*}
are not true.

\section{Rank one direct product case $\ZZ_2^J$}\label{sec:product}

The rank one direct product case $\ZZ_2^J$ is an important case of Dunkl analysis, developing intensely in recent years, see e.g. 
\cite{Del, Dz, NS1, NS2, NSS, TX, Xu}. In these papers, the root systems $B_1\times\cdots B_1$ are commonly considered. An explicit formula for the intertwining operator in the latter (non invariant) case was obtained in \cite{Xu}, and generalized in \cite{MYoussfi2} in the case of orthogonal positive root systems. In our approach, we consider equivalently the root systems $A_1\times\cdots A_1$, and the intertwining operator formula obtained in \cite{Sawyer} for the $A_1$ root system. This leads to estimates of the $P^W$ and $N^W$ kernels both in $A_1\times\cdots A_1$ and $B_1\times\cdots B_1$ cases.

The dual Abel transform (i.e.  the $W$-invariant intertwining operator) for $A_1$ can be written as
\begin{align*}
\mathcal{A}(f)(y)=\frac{\Gamma(m)}{(\Gamma(m/2))^2}
\,(y_1-y_2)^{1-m}\,\int_{y_2}^{y_1}\,f(z,y_1+y_2-z)\,((y_1-z)\,(z-y_2))^{m/2-1}\,dz
\end{align*} 
(see \cite{Sawyer}).

We now consider $J$ orthogonal roots in $\RR^d$. The dual of the generalized Abel transform $\overbrace{A_1\times\cdots A_1}^J$ can be written as
\begin{align*}
\mathcal{A}^*(f)(y)&=\left(\prod_{i=1}^J\,\frac{\Gamma(m_i)}{(\Gamma(m_i/2))^2}\right)
\,\prod_{i=1}^J\,(y^{(i)}_1-y^{(i)}_2)^{1-m_i}
\\&\quad
\,\int_{y^{(J)}_2}^{y^{(J)}_1}\dots\,
\int_{y^{(1)}_2}^{y^{(1)}_1}
\,f\left((z^{(i)},y^{(i)}_1+y_2^{(i)}-z^{(i)})_{i=1,\dots,J}\right)
\\&\quad\,\prod_{i=1}^J\,((y^{(i)}_1-z)\,(z-y^{(i)}_2))^{m_i/2-1}\,dz^{(1)}\,\cdots\,dz^{(J)}
\end{align*}
where $m_i$ is the multiplicity of the root acting on $y^{(i)}_1$ and $y^{(i)}_2$.

\begin{remark}
One may have the root system $A_1$ in $\R$ (or in $\R^2$ with $y_1+y_2=0$) or in $\R^d$ where the roots only act on, say, $y_1$ and $y_2$. The same comment applies to $\overbrace{A_1\times\cdots A_1}^J$ in $\R^d$. To minimize the notation, we will only indicate the variables on which the roots act.
\end{remark}

We have
\begin{align*}
\frac{P^W(x,y)}{1-|x|^2}
&=C\,\mathcal{A}^*\left(\frac{1}{(|x-y|^2+2\,\langle x,y-\cdot\rangle)^{d/2+(m_1+\dots+m_J)/2}}\right)(y)\\
&=C'\,\prod_{i=1}^J\,(y^{(i)}_1-y^{(i)}_2)^{1-m}
\\&\quad
\,\int_{y^{(J)}_2}^{y^{(J)}_1}\dots\,
\int_{y^{(1)}_2}^{y^{(1)}_1}
\frac{\prod_{i=1}^J\,((y^{(i)}_1-z^{(i)}))\,(z^{(i)})-y^{(i)}_2))^{m_i/2-1}\,dz^{(1)}\,\cdots\,dz^{(J)}}
{\left(|x-y|^2+2\,\sum_{i=1}^J\,(x^{(i)}_1-x^{(i)}_2)\,(y^{(i)}_1-z^{(i)})\right)^{d/2+(m_1+\dots+m_J)/2}}
\end{align*}
where $m_i>0$ is the root multiplictiy for the $i$-th factor.

With the change of variables, $v^{(i)}=\frac{y^{(i)}_1-z^{(i)}}{y^{(i)}_1-y^{(i)}_2}$,
we find that
\begin{align}
\frac{P^W(x,y)}{1-|x|^2}
&=C'\,\int_{0}^1\dots\,\int_0^1
\,\frac{\prod_{i=1}^J\,(v^{(i)}\,(1-v^{(i)}))^{m_i/2-1}\,dv^{(1)}\,\cdots\,dv^{(J)}}{\left(|x-y|^2+2\,\sum_{i=1}^J\,(x^{(i)}_1-x^{(i)}_2)\,(y^{(i)}_1-y^{(i)}_2)\,v^{(i)}\right)^{d/2+(m_1+\dots+m_J)/2}}.\label{simple}
\end{align}

In a similar fashion, for $d\geq2$, we have
\begin{align*}
N^W(x,y)
&=C\,\mathcal{A}^*\left(\frac{1}{(|x-y|^2+2\,\langle x,y-\cdot\rangle)^{d/2-1+(m_1+\dots+m_J)/2}}\right)(y)\\
&=C'\,\int_{0}^1\dots\,\int_0^1
\,\frac{\prod_{i=1}^J\,(v^{(i)}\,(1-v^{(i)}))^{m_i/2-1}\,dv^{(1)}\,\cdots\,dv^{(J)}}{\left(|x-y|^2+2\,\sum_{i=1}^J\,(x^{(i)}_1-x^{(i)}_2)\,(y^{(i)}_1-y^{(i)}_2)\,v^{(i)}\right)^{d/2-1+(m_1+\dots+m_J)/2}}.
\end{align*}

We want to prove Conjecture \ref{TL} and Conjecture \ref{TLN} for the root system $\overbrace{A_1\times\cdots A_1}^J$ using induction on $J\geq 1$. To accomplish this, we will need the following technical result.
\begin{lemma}\label{reduction}
Suppose $A>0$ and $B>0$. Then
\begin{align*}
T=\int_0^1\,\frac{(u\,(1-u))^{m/2-1}\,du}{(A+B\,u)^M} 
\left\lbrace\begin{array}{ll}
\asymp \frac{1}{A^{M-m/2}\,(A+B)^{m/2}}&\hbox{if $M>m/2$},\\
\asymp\frac{\ln\left(2\,\frac{A+B}{A}\right)}{(A+B)^{m/2}}&\hbox{if $M=m/2$}.
\end{array}
\right.
\end{align*}
\end{lemma}

\begin{proof}
If $A\geq B/4$ then 
\begin{align*}
\int_0^1\,\frac{(u\,(1-u))^{m/2-1}\,du}{(A+B\,u)^M}
\asymp \int_0^1\,\frac{(u\,(1-u))^{m/2-1}\,du}{A^M}
\asymp \frac{1}{A^M}.
\end{align*}

If $A\leq B/4$ then 
\begin{align*}
T=\int_0^{A/B}+\int_{A/B}^{1/2}+\int_{1/2}^1=T_1+T_2+T_3.
\end{align*}

Now,
\begin{align*}
T_1\asymp \int_0^{A/B}\,\frac{u^{m/2-1}\,du}{A^M}=\frac{2}{m}\,(A/B)^{m/2}\,\frac{1}{A^M}
\asymp\frac{1}{A^{M-m/2}\,B^{m/2}}
\end{align*}
since $u\leq A/B$ and $A\leq A+B\,u\leq A+B\,A/B=2\,A$. Next, we note that
\begin{align*}
T_2\asymp \int_{A/B}^{1/2}\,\frac{u^{m/2-1}\,du}{(B\,u)^M}
&=\frac{1}{B^M}\, \int_{A/B}^{1/2}\,u^{m/2-M-1}\,du\\
&=\frac{1}{(M-m/2)\,B^M}\,((A/B)^{m/2-M}-(1/2)^{m/2-M})\\
&=\frac{1}{M-m/2}\,\left(\frac{1}{A^{M-m/2}\,B^{m/2}}-\frac{(1/2)^{m/2-M}}{B^M}\right)
\end{align*}
if $M>m/2$ since $A/B\leq u\leq 1/2$ and $B\,u\leq A+B\,u\leq 2\,B\,u$. If $M=m/2$ then
\begin{align*}
T_2\asymp \int_{A/B}^{1/2}\,\frac{u^{m/2-1}\,du}{(B\,u)^{m/2}}
&=\frac{1}{B^{m/2}}\, \int_{A/B}^{1/2}\,u^{-1}\,du
=\frac{1}{B^{m/2}}\, \ln((B/A)/2).
\end{align*}
Finally,
\begin{align*}
T_3\asymp \int_{1/2}^1\,\frac{(1-u)^{m/2-1}\,du}{B^M}
&\asymp\frac{1}{B^M}
\end{align*}
since $u\geq 1/2$ and $B/2\leq A+B\,u\leq 2\,B$. The result follows then easily.
\end{proof}

\begin{theorem}\label{Pprod}
We consider the case $\overbrace{A_1\times\cdots A_1}^J$ . Suppose $x$, $y\in\overline{\a^+}$. We have
\begin{align*}
\frac{P^W(x,y)}{1-|x|^2}
\asymp \frac{1}{|x-y|^d\,\prod_{i=1}^J\,|x-\sigma^{(i)}_\alpha\,y|^{m_i}}
\end{align*}
where $\sigma^{(i)}_\alpha$ is the reflection with respect to $\alpha$ acting on the variables $y^{(i)}_1$ and $y^{(i)}_2$.
\end{theorem}

\begin{proof}
We have $|x-\sigma_\alpha\,y|^2\asymp |x-y|^2+2\,(x_1-x_2)\,(y_1-y_2)$ (refer to Lemma \ref{basic simple}).

We use induction on $J\geq 1$. When $J=1$, the result follows using \eqref{simple}, Lemma \ref{reduction} with $A=|x-y|^2$, $B=2\,(x^{(1)}_1-x^{(1)}_2)\,(y^{(1)}_1-y^{(1)}_2)$ and $M=d/2+m_1/2$.

Assume that the result holds for $J-1$, $J\geq 2$. There exists $j$ such that $(x^{(i)}_1-x^{(i)}_2)\,(y^{(i)}_1-y^{(i)}_2)
\leq (x^{(j)}_1-x^{(j)}_2)\,(y^{(j)}_1-y^{(j)}_2)$ for all $i$. To fix things, say $j=J$.

Use \eqref{simple}, Lemma \ref{reduction} with $A=|x-y|^2+2\,\sum_{i=1}^{J-1}\,(x^{(i)}_1-x^{(i)}_2)\,(y^{(i)}_1-y^{(i)}_2)\,v^{(i)}$, 
$B=2\,(x^{(J)}_1-x^{(J)}_2)\,(y^{(J)}_1-y^{(J)}_2)$ and $M=d/2+(m_1+\dots+m_J)/2$, integrating with respect to $v^{(J)}$. We find that
\begin{align*}
\lefteqn{\frac{P^W(x,y)}{1-|x|^2}}\\
&\asymp\int_{0}^1\dots\,\int_0^1
\,\frac{\prod_{i=1}^J\,(v^{(i)}\,(1-v^{(i)}))^{m_i/2-1}\,dv^{(1)}\,\cdots\,dv^{(J-1)}}
{\left(|x-y|^2+2\,\sum_{i=1}^{J-1}\,(x^{(i)}_1-x^{(i)}_2)\,(y^{(i)}_1-y^{(i)}_2)\,v^{(i)}\right)^{d/2+(m_1+\dots+m_{J-1})/2}}
\\\qquad&
\frac{1}
{\left(|x-y|^2+2\,\sum_{i=1}^{J-1}\,(x^{(i)}_1-x^{(i)}_2)\,(y^{(i)}_1-y^{(i)}_2)\,v^{(i)}+2\,(x^{(J)}_1-x^{(J)}_2)\,(y^{(J)}_1-y^{(J)}_2)\right)^{m_J/2}}
\\&\asymp
\frac{1}
{\left(|x-y|^2+2\,(x^{(J)}_1-x^{(J)}_2)\,(y^{(J)}_1-y^{(J)}_2)\right)^{m_J/2}}
\\\qquad&
\,\int_{0}^1\dots\,\int_0^1
\,\frac{\prod_{i=1}^J\,(v^{(i)}\,(1-v^{(i)}))^{m_i/2-1}\,dv^{(1)}\,\cdots\,dv^{(J-1)}}
{\left(|x-y|^2+2\,\sum_{i=1}^{J-1}\,(x^{(i)}_1-x^{(i)}_2)\,(y^{(i)}_1-y^{(i)}_2)\,v^{(i)}\right)^{d/2+(m_1+\dots+m_{J-1})/2}}\\
\\&\asymp 
\frac{1}{|x-y|^d\,\prod_{i=1}^J\,|x-\sigma^{(i)}_\alpha\,y|^{m_i}}
\end{align*}
(the last step follows using the induction hypothesis).
\end{proof}

\begin{theorem}
We consider the case $\overbrace{A_1\times\cdots A_1}^J$ . Suppose $x$, $y\in\overline{\a^+}$. For $d\geq 3$ we have
\begin{align*}
N^W(x,y)
\asymp \frac{1}{|x-y|^{d-2}\,\prod_{i=1}^J\,|x-\sigma^{(i)}_\alpha\,y|^{m_i}}.
\end{align*}
\end{theorem}

\begin{proof}
The proof is basically the same as for Theorem \ref{Pprod}.
\end{proof}

We proceed with handling the case $d=2$.

\begin{proposition}
The Newton kernel in the case of $A_1$ in $\R^2$ satisfies
\begin{align}
N^W(x,y)
\asymp \frac{\ln\,\left(2\,\frac{|x-\sigma_\alpha\,y|^2}{|x-y|^2}\right)}{|x-\sigma_\alpha\,y|^m}.\label{A1gen}
\end{align}
The Newton kernel in the case of $A_1\times A_1$ in $\R^2$ satisfies 
\begin{align}
\frac{\ln\,\left(2\,\frac{|x-\sigma^{(1)}_\alpha\,y|\wedge |x-\sigma^{(2)}_\alpha\,y|}{|x-y|^2}\right)}{|x-\sigma^{(1)}_\alpha\,y|^{m_1}\,|x-\sigma^{(2)}_\alpha\,y|^{m_2}}\label{A1gen2}
\end{align}
where $a\wedge b=\min\{a,b\}$.
\end{proposition}

\begin{proof}
Equation \eqref{A1gen} follows directly from Lemma \ref{reduction} with $M=m_1/2$.

To obtain equation \eqref{A1gen2}, we can assume without loss of generality that $(x^{(1)}_1-x^{(1)}_2)\,(y^{(1)}_1-y^{(1)}_2)\leq (x^{(2)}_1-x^{(2)}_2)\,(y^{(2)}_1-y^{(2)}_2)$. Apply Lemma \ref{reduction} to the integral representing $N^W(x,y)$ with $M=m_1+m_2$, $A=|x-y|^2+2\,(x^{(1)}_1-x^{(1)}_2)\,(y^{(1)}_1-y^{(1)}_2)\,u^{(1)}$ and 
$B=2\,(x^{(2)}_1-x^{(2)}_2)\,(y^{(2)}_1-y^{(2)}_2)$, integrating with respect to $u^{(2)}$. We obtain
\begin{align*}
N^W(x,y)
\asymp\frac{1}{|x-\sigma^{(2)}_\alpha\,y|^{m_2}}
\,\int_0^1\,\frac{(u^{(1)}\,(1-u^{(1)}))^{m/2-1}\,du^{(1)}}{(|x-y|^2+2\,(x^{(1)}_1-x^{(1)}_2)\,(y^{(1)}_1-y^{(1)}_2)\,u^{(1)})^{m/2}}.
\end{align*}

If we apply Lemma \ref{reduction} with $M=m_1$, we get
\begin{align*}
N^W(x,y)
\asymp\frac{1}{|x-\sigma^{(2)}_\alpha\,y|^{m_2}}
\,\frac{\ln\,\left(2\,\frac{|x-\sigma^{(1)}_\alpha\,y|^2}{|x-y|^2}\right)}{|x-\sigma^{(1)}_\alpha\,y|^m}.
\end{align*}
\end{proof}

We end this Section by a formulation of the results for the systems 
$\overbrace{B_1\times\cdots B_1}^J$ acting on $\R^d,\, d\ge J$, i.e. the symmetries are $\sigma_\beta^{(i)}(y)=(y_1,\ldots,-y_i,y_{i+1},\ldots, y_d)$ for $i\le J$.
\begin{corollary}
Consider the case $\overbrace{B_1\times\cdots B_1}^J$.   Suppose $x$, $y\in\overline{\a^+}$.
\begin{enumerate}
\item We have
\begin{align*}
\frac{P^W(x,y)}{1-|x|^2}
\asymp \frac{1}{|x-y|^d\,\prod_{i=1}^J\,|x-\sigma^{(i)}_\beta\,y|^{m_i}}
\end{align*}

\item For $d\geq 3$, we have
\begin{align*}
N^W(x,y)
\asymp \frac{1}{|x-y|^{d-2}\,\prod_{i=1}^J\,|x-\sigma^{(i)}_\beta\,y|^{m_i}}.
\end{align*}
\end{enumerate}
\end{corollary}

\section{Applications to stochastic processes}\label{stochastic}

From a probabilistic point of view, the formula \eqref{operator}
\begin{equation*}
\Delta^W f= \pi^{-1}\, \Delta^{\R^d} (\pi\, f),
\end{equation*}
gives the generator of the Doob $h$-transform (refer to \cite{ry}) of the Brownian Motion on $\R^d$ with the excessive function $h(x)=\pi(x)$. For the root system $A_{d-1}$ on $\R^d$, the operator $\Delta^W$ is the generator of the Dyson Brownian Motion on $\R^d$ (\cite{Dyson}), {\it i.e.} the $d$ Brownian independent particles $B^{(1)}_t,\ldots,B^{(d)}_t $ conditioned not to collide.
More generally, for any root system $\Phi$ on $\R^d$, the construction of a Dyson Brownian Motion $D_t^\Phi$ as a Brownian Motion conditioned not to touch the walls of the positive Weyl chamber, can be done (\cite{Grabiner}).
It is known that $\pi$ is $\Delta_{\R^d}$-harmonic on $\R^d$ (\cite{Grabiner}), so in particular $\pi$ is excessive.

Dyson Brownian Motion $D_t^\Phi$
is one of the most important models of non-colliding particles (see e.g. \cite{AK, AV, Katori}).
In \cite{AV}, $W$-invariant Dunkl processes are called  multivariate Bessel processes.

The only difference with the symmetric flat complex case is that no invariant measure $\pi^2(y)\,dy$ appears for the integral kernels in the Dyson Brownian Motion case.

Theorem \ref{main k1}   implies estimates for the Poisson kernel $P^D$ and the Newton kernel $N^D$ for the Dyson Brownian Motion. These estimates 
are essential for the potential theory of the process $D_t^\Phi$ and, consequently, of non-colliding stochastic particles.
 
\begin{corollary}
\label{Dyson}
1.
For $x\in B^+$ and $y\in S^+$ we have
\begin{align*}
P^D(x,y)\asymp \frac{(1-|x|^2)\pi^2(y)}{|x-y|^d\prod_{\alpha\in \Phi_+} |x-\sigma_\alpha y|^{2 }}
\end{align*}

2. 
For $x, y\in \overline{\a^+}$ and $d\geq 3$, we have
\begin{align*}
N^D(x,y)\asymp \frac{\pi^2(y)}{|x-y|^{d-2}\,\prod_{\alpha\in \Phi_+}|x-\sigma_\alpha y|^{2}}.
\end{align*}

\end{corollary}

\begin{acknowledgment}
P. Graczyk was partially supported by Labex CHL Lebesgue and Programme R\'egional DEFIMATHS. T. Luks is grateful to Universit\'e d'Angers and to Laurentian University for their hospitality. P. Sawyer was supported by Laurentian University. He is thankful to Universit\'e d'Angers and to Universit\"at Paderborn for their hospitality.
\end{acknowledgment}


\end{document}